\documentclass[12pt,reqno]{amsart}
\usepackage{amssymb}
\usepackage{amsmath}
\usepackage{enumitem}
\usepackage{mathrsfs}
\usepackage[all]{xy}
\usepackage{hyperref}
\usepackage{marginnote}

\usepackage[margin=1.5in]{geometry}

\RequirePackage{mathrsfs} %\let\mathcal\mathcal
\usepackage[OT2,T1]{fontenc}
\DeclareSymbolFont{cyrletters}{OT2}{wncyr}{m}{n}
\DeclareMathSymbol{\Sha}{\mathalpha}{cyrletters}{"58}

\newtheorem{theorem}{Theorem}[section]
\newtheorem{lemma}[theorem]{Lemma}
\newtheorem{proposition}[theorem]{Proposition}
\newtheorem{corollary}[theorem]{Corollary}
\newtheorem{conjecture}[theorem]{Conjecture}

\theoremstyle{definition}
\newtheorem*{ack}{Acknowledgements}
\newtheorem*{con}{Conventions}
\newtheorem{remark}[theorem]{Remark}

\newtheorem{definition}[theorem]{Definition}

\numberwithin{equation}{section} \numberwithin{figure}{section}

\DeclareMathOperator{\Pic}{Pic} 
\DeclareMathOperator{\Gal}{Gal} 
\DeclareMathOperator{\Aut}{Aut} 

\DeclareMathOperator{\Spec}{Spec}

\DeclareMathOperator{\an}{an}

\DeclareMathOperator{\rank}{rank}
\DeclareMathOperator{\prim}{prim}
\DeclareMathOperator{\Hom}{Hom} 
\DeclareMathOperator{\im}{Im}

\DeclareMathOperator{\Br}{Br}

\DeclareMathOperator{\Isom}{Isom} \DeclareMathOperator{\Inn}{Inn}
\DeclareMathOperator{\Hilb}{Hilb}
\DeclareMathOperator{\Lin}{Lin}

\newcommand{\PGL}{\textrm{PGL}}

\newcommand{\PO}{\textrm{PO}}

\newcommand\PP{\mathbb{P}}

\newcommand\ZZ{\mathbb{Z}}
\newcommand\NN{\mathbb{N}}
\newcommand\QQ{\mathbb{Q}}
\newcommand\RR{\mathbb{R}}
\newcommand\CC{\mathbb{C}}

\newcommand\OO{\mathcal{O}}

\title[Complete intersections]{Complete intersections: Moduli, Torelli, and good reduction}

\author{A. Javanpeykar}
\address{A. Javanpeykar \\
Institut f\"{u}r Mathematik\\
Johannes Gutenberg-Universit\"{a}t Mainz\\
Staudingerweg 9, 55099 Mainz\\
Germany.}
\email{peykar@uni-mainz.de}

\author{D. Loughran}
\address{D. Loughran \\
Leibniz Universit\"{a}t Hannover,
Institut f\"{u}r Algebra, Zahlentheorie
    und Diskrete Mathematik\\
Welfengarten 1\\
30167 Hannover\\
Germany.}
\email{loughran@math.uni-hannover.de}

\subjclass[2010]
{11G35  %Varieties over global fields
(14M10, %Complete intersections
14K30,  %Picard schemes, higher Jacobians 
14J50,  %Automorphisms of surfaces and higher-dimensional varieties
14C34,  %Torelli problem 
14D23)}  %Stacks and moduli problems

\begin{document}
 
\begin{abstract}
	We study the arithmetic of complete intersections in projective space over number fields.
	Our main results include arithmetic Torelli theorems and versions of the Shafarevich conjecture, as proved
	for curves and abelian varieties by Faltings.
	For example, we prove an analogue of the Shafarevich conjecture for cubic and quartic threefolds
	and intersections of two quadrics.
\end{abstract}

\maketitle
\tableofcontents

\thispagestyle{empty}

\section{Introduction}
An important finiteness statement in algebraic number theory is the theorem
of Hermite and Minkowski: there are only finitely many number fields
of bounded degree over $\QQ$ which are unramified outside of a given finite set $S$ of finite places of $\QQ$.
Shafarevich was the first to notice that such finiteness statements occur elsewhere, and at the 1962 ICM in Stockholm \cite{Shaf1962} he formulated one of his most famous conjectures: given $g \geq 2$
and a finite set $S$ of finite places of a number field $K$, the set of $K$-isomorphism
classes of smooth projective curves of genus $g$ over $K$ with good reduction outside of $S$
is finite. Faltings proved this conjecture in his paper on Mordell's conjecture \cite{Faltings2},
and also proved the analogous finiteness statement for abelian varieties.

It is natural to ask for which other classes of varieties such finiteness statements hold.
Analogous results have been proven in the following cases:
\begin{itemize}
	\item Polarised K$3$ surfaces of bounded degree and cubic fourfolds \cite{Andre}.
	\item Del Pezzo surfaces \cite{Scholl}.
	\item Flag varieties \cite{JL2}.
	\item Certain surfaces of general type \cite{Jav15}.
\end{itemize}
Part of the aim of this paper is to illustrate that such statements should be rife in arithmetic geometry. 
In particular, we prove an analogous finiteness result for certain complete intersections in projective space.

\begin{theorem}\label{theorem: main theorem}
Let $K$ be a number field, let $S$ be a finite set of finite places of $K$, let $n \geq 1$ and let $b\geq 0$. Then the set of $K$-linear isomorphism classes of $n$-dimensional complete intersections over $K$
with Hodge level at most $1$, whose $n$th Betti number equals $b$ and with good reduction outside $S$, is finite.
\end{theorem}
Here by \emph{good reduction}, we mean good reduction as a complete intersection; see Definition \ref{def:good_ceduction_CI}.
A \emph{linear} isomorphism is an isomorphism which is induced by an automorphism of the ambient projective space.
The \emph{Hodge level} is a certain element of $\ZZ \cup \{-\infty\}$ that one associates to a smooth complete intersection,
defined in terms of its Hodge structure, which gives a rough measure of its geometrical complexity.
The complete intersections of Hodge level at most $1$ have been completely classified,
and this classification, together with relevant definitions, can be found in Section \ref{sec: Hodge theory}.
Note that the classification implies that if $n > 1$, then there are only finitely many
choices for the $n$th Betti number, in particular, the assumption on the Betti number in Theorem \ref{theorem: main theorem}
is only required when $n=1$.
Interesting new examples to which Theorem \ref{theorem: main theorem} applies include cubic and quartic threefolds 
and intersections of two quadrics. 
Many  of our results generalise from number fields to finitely generated fields of characteristic zero; 
see Section \ref{Sec:Shaf} for our most general results.

Let us explain the ideas behind the proof of Theorem \ref{theorem: main theorem}. The proof proceeds by handling various cases in turn, ordered by their Hodge level. The complete intersections of Hodge level $-\infty$
are quadric hypersurfaces. In particular, the result here follows from our general result on flag varieties
\cite{JL2}; the key finiteness property required being the finiteness of the Tate-Shafarevich set of a linear algebraic group over $K$. 

In Hodge level $0$ the key new case is that of intersections of two quadrics (Scholl's result \cite{Scholl} 
already handles the case of cubic surfaces). Here we use the associated pencil
of quadrics to reduce the problem to the finiteness of the set
of solutions of $S$-unit equations.  The proofs of the Shafarevich conjecture for 
elliptic curves (see \cite{Parshin} and \cite[Thm.~6.1]{Silverman}), cyclic curves \cite{dJRe,JvK},
and cubic surfaces \cite{Scholl} also reduce to such finiteness statements.

The case of Hodge level $1$ is the deepest.
Here we use the theory of the \emph{intermediate Jacobian}, which is a higher-dimensional analogue of the Jacobian of a curve. This is usually 
constructed complex analytically, however Deligne \cite{Del72} has shown in our case how to make this theory work over any field of characteristic $0$, which is of course crucial for arithmetic applications. 
Faltings used his finiteness result for abelian varieties to deduce the Shafarevich conjecture for curves via a Torelli theorem (namely, the fact that any curve of genus $g \geq 2$ over $K$
is uniquely determined up to $K$-isomorphism by its Jacobian, see e.g.~\cite[Cor.~VII.12.2]{CornellSilverman}).
We follow a similar strategy, and use Faltings's theorem to reduce to showing the following ``arithmetic Torelli theorem''.
\begin{theorem} \label{thm:arithmetic_torelli}
	Let $K$ be a field of characteristic $0$ and 
	let $A$ be a principally polarised abelian variety over $K$.
	Then the set of $K$-linear isomorphism classes of smooth complete intersections $X$ of Hodge level $1$ over $K$,
	which are not intersections of two quadrics nor curves of genus $1$, and whose intermediate Jacobian is
	isomorphic to $A$ as a principally polarised abelian variety, is finite.
\end{theorem}
Note that our result is slightly weaker than the known Torelli theorem for curves.
This is due to the fact that one does not even know a full global Torelli over $\CC$
in all the cases of interest (e.g.~this is unknown for quartic threefolds over $\CC$).
Theorem \ref{thm:arithmetic_torelli} is however sufficient for the proof of Theorem~\ref{theorem: main theorem}. The restriction to avoid intersections of two quadrics and curves of genus $1$ is genuinely required; we show that the analogous statement  is false in these cases in Section \ref{section:intnsof2}.

The proof of Theorem \ref{thm:arithmetic_torelli} requires numerous inputs from geometry.
The key ones being an infinitesimal Torelli theorem over $\CC$ due to Flenner \cite{Fl86}, together with 
the fact that the automorphism group of such a smooth complete intersection acts faithfully on its cohomology (see Section \ref{sec:automorphisms_faithful}).   We in fact
show stronger results than stated here, namely Theorem \ref{thm:arithmetic_torelli} is proved
by showing that the intermediate Jacobian
gives rise to a separated representable quasi-finite morphism of stacks; see Section \ref{Sec:Arithmetic}
for our complete results.

One knows full global Torelli theorems over $\CC$ for cubic threefolds \cite{CG72}
and odd-dimensional intersections of three quadrics \cite[Cor.~4.5]{Deb89}.
We are able to extend these to any field of characteristic $0$.

\begin{theorem} \label{thm:global_arithmetic_torelli}
	Let $K$ be a field of characteristic $0$ and 
	let $X_1$ and $X_2$ be either smooth cubic threefolds or smooth odd-dimensional 
	complete intersections of three quadrics over $K$. 
	
	If $J(X_1) \cong J(X_2)$ as principally
	polarised abelian varieties, then $X_1 \cong X_2$.
\end{theorem}
We emphasise that this does not follow formally from the Torelli theorem over $\CC$;
indeed, one knows a global Torelli theorem over $\CC$ for odd-dimensional intersections of two quadrics \cite[Cor.~3.4]{Don80},
yet these do not satisfy a global Torelli theorem over every field of characteristic $0$, as we show in Section \ref{section:intnsof2}.

The authors believe that Theorem \ref{theorem: main theorem} 
should be a special case of a more general finiteness statement.
\begin{conjecture}[Shafarevich conjecture for complete intersections]\label{conj} 
Let $K$ be a number field, let $S$ be a finite set of finite places of $K$
and let $T$ be a type. Then the set of $K$-linear isomorphism classes of smooth complete intersections of type $T$ over $K$
 with good reduction outside $S$ is finite.
\end{conjecture}

Here a type is just a collection $T=(d_1,\ldots,d_c;n)$ which specifies the dimension of the complete intersection
and the degrees of its defining equations (see Section \ref{section:complete_basics}). Note that
Theorem \ref{theorem: main theorem} proves Conjecture \ref{conj} for complete intersections of Hodge level at most $1$. 

Our last main theorem shows that Conjecture \ref{conj} follows from the Lang--Vojta conjecture in many cases
(see Section \ref{section:lang}  for a discussion of this conjecture). To avoid certain technical difficulties, we only prove it for  hypersurfaces and complete intersections of general type. 

\begin{theorem}\label{thm: lang implies shaf}
The Lang--Vojta conjecture implies the Shafarevich conjecture for hypersurfaces and complete intersections of general type.
\end{theorem}

The Lang--Vojta conjecture enters the picture via a relationship, which we often exploit in this paper,
between complete intersections with good reduction and integral points on appropriate moduli stacks. To prove Theorem \ref{thm: lang implies shaf}, we construct a finite    \'{e}tale cover of these moduli stacks using a ``level structure'' (see Theorem \ref{thm: finet} for a precise statement). This uses the fact that the automorphism groups of many complete intersections act faithfully on their cohomology groups (Proposition \ref{prop:faithful}). Once we attach level structure we obtain a scheme whose subvarieties are of log-general type by a result of Zuo \cite{ZuoNeg};
Theorem \ref{thm: lang implies shaf} then follows from the Lang--Vojta conjecture via a descent argument, similar to the theorem of Chevalley-Weil \cite[\S4.2]{Ser97}.

\subsection*{Outline of the paper} In Section \ref{Sec:Complete}, we study the geometry of complete intersections,
in particular their moduli stacks, automorphism groups and Hodge theory. Our main result here is
Theorem~\ref{thm:qf_torelli}, a ``quasi-finite Torelli theorem'' for smooth 
complete intersections over $\CC$, under suitable assumptions.

In Section \ref{Sec:Arithmetic} we give arithmetic applications of these results by proving Theorem \ref{thm:arithmetic_torelli} and Theorem \ref{thm:global_arithmetic_torelli}, together with results on the associated moduli stacks. In Section \ref{Sec:Good}, we define the notion of good reduction for complete intersections and study some of its basic properties. Our main result here is Theorem \ref{thm: twists}, which says that a complete intersection admits only finitely many twists with good reduction, under suitable assumptions.

Section \ref{Sec:Shaf} is dedicated to the proof of Theorem \ref{theorem: main theorem}, and in Section \ref{section:Lang_Shaf}
we prove Theorem \ref{thm: lang implies shaf}. We actually prove a stronger statement (Theorem~\ref{thm: lis})
which applies over arithmetic schemes. This proof requires various results on the moduli stack of complete intersections of certain types, in particular we show in Theorem \ref{thm: finet} how to attach a level structure to these stacks.

\begin{ack} 
We are grateful to Giuseppe Ancona,  Jean-Beno\^it Bost, Martin Bright, Fr\'ederic Campana,  Jean-Louis Colliot-Th\'el\`ene, Bas Edixhoven, Jochen Heinloth, Marc Hindry, David Holmes,  Ben Moonen, Laurent Moret-Bailly,  Duco van Straten, Lenny Taelman, Olivier Wittenberg, and Kang Zuo for helpful discussions on several parts of this paper. 
Special thanks go to Olivier Benoist for very helpful discussions on complete intersections and level structure, Yohan Brunebarbe for his help on period maps and the proof of Theorem \ref{thm:qf_torelli}, and  Angelo Vistoli for answering our questions on stacks and his help in proving Proposition \ref{prop:crit_for_rep}. We are very grateful to the anonymous referee for useful comments.
The first named author gratefully acknowledges the support of SFB/Transregio 45.
\end{ack}

\begin{con} 
For a number field $K$, we let $\OO_K$ denote its ring of integers. If $S$ is a finite set of finite places of $K$, we let 
$\OO_K[S^{-1}]$
denote the ring of $S$-integers of $K$.

A variety over a field $k$ is a   finite type  $k$-scheme. For a Noetherian scheme $X$,
we denote by $X^{(1)}$ the set of points  of $X$ of codimension $1$.

A \textit{Dedekind scheme} is an integral normal Noetherian  one-dimensional scheme. 
An \textit{arithmetic scheme} is an integral regular finite type  flat scheme over $\ZZ$.  Note that if $B$ is a one-dimensional arithmetic scheme, then there exist a number field $K$ and a finite set of finite places $S$ of $K$ such that $B \cong \Spec \OO_{K}[S^{-1}]$.

Let $G$ be a smooth group scheme of finite type over a scheme $B$.
We denote by $\mathrm{H}^1(B,G)$ the first \v{C}ech cohomology set of $B$ with coefficients in $G$
with respect to the \'etale topology \cite[\S III.2]{MilneEC}.

For a stack $\mathcal{M}$ and a scheme $B$, we denote by $[\mathcal{M}(B)]$ the set
of $B$-isomorphism classes of objects of the groupoid $\mathcal{M}(B)$.   This is also sometimes denoted by $\pi_0(\mathcal M(B))$ in 
the literature.

For $g \in \NN$, we denote by $\mathcal{A}_{g,1}$ the moduli stack of principally polarised abelian varieties of relative dimension $g$ over $\ZZ$.

We always consider complete intersections of codimension $c$ in $\PP^{n+c}$ of type 
$$T  =(d_1,\ldots,d_c;n) \qquad \text{(see Definition \ref{def:type})}.$$
For such a complete intersection $X$, we denote by $\Lin X$ the group scheme of \emph{linear}
automorphisms of $X$ (i.e.~those automorphisms which are induced by an automorphism
of the ambient projective space).
\end{con}

\section{The geometry of complete intersections} \label{Sec:Complete}
In this section we gather various facts we shall need about the
geometry of complete intersections, in particular, their Hodge theory,
the intermediate Jacobian, their automorphisms and the structure of their moduli spaces. \\

\subsection{Complete intersections}\label{section:complete_basics}
\subsubsection{Definitions}
Let $B$ be a scheme.

\begin{definition} \label{def:type}
    A \emph{type} is a collection of integers 
    $$T=(d_1,\ldots,d_c;n)  \quad \mbox{with} \quad n \geq 1, \quad c\geq1, \quad 2\leq d_1\leq \ldots \leq d_c.$$
    A \textit{complete intersection of type $T$ over $B$} is a closed subscheme of codimension $c$ in $\PP^{n+c}_B$ 
    that is flat over $B$ and which is the zero locus of $c$ homogeneous polynomials of degrees $d_1,\ldots,d_c$
    over $B$, respectively. A complete intersection over $B$ is a complete intersection of unspecified type.
\end{definition}
Here, a homogeneous polynomial of degree $d$ over a scheme $B$ means a global section of the sheaf
$\pi_* \mathcal{O}_{\smash{\PP^{n+c}_B}}(d)$, where $\pi$ is the projection to $B$.
We reserve the variables $d_1,\ldots,d_c$ and $n$ for the above usage throughout this paper.

We will say that a type $T= (d_1,\ldots,d_c;n)$ is a hypersurface if $c=1$.
We will say that $T$ is of general type if $d_1+\ldots+ d_c \geq n+c+2$ (this agrees with the usual definition,
as the adjunction formula implies that such complete intersections are exactly those with ample canonical bundle).

\subsubsection{The  moduli stack of smooth complete intersections}\label{section: moduli stack}  
In this paper we will require various facts on the moduli stack
of smooth complete intersections. The relevant  theory has been worked out in great detail by Olivier Benoist
in his Ph.D. thesis \cite{BenoistThesis};  many of the relevant results for us appear
in \cite{Ben13}. 
 
We will assume that the reader is familiar with the basic theory of stacks, as can be found for example in \cite{LMB}.
Among some of the  basic definitions and results we need are  \cite[Def.~3.9]{LMB}, \cite[Def.~4.1]{LMB}, \cite[Def.~4.7.1]{LMB}, \cite[Def.~7.6]{LMB}, \cite[Lem.~7.7]{LMB}, and \cite[Thm.~8.1]{LMB}.

We first recall the construction of the moduli stack of smooth complete intersections. Let $T$ be a type
and let $\mathrm{Hilb}_T$ denote the open subscheme of the Hilbert scheme of $\PP_\ZZ^{n+c}$
which parametrises  smooth complete intersections of type $T$.
We define the moduli stack $\mathcal{C}_T$ of smooth complete intersections
of type $T$ to be the quotient stack $[ \PGL_{n+c+1} \backslash \mathrm{Hilb}_T ]$.
For a scheme $B$, we write $\mathcal C_{T,B}$ for $\mathcal C_T \times_\ZZ B$.

\begin{proposition}[Benoist]\label{prop:moduli_stack}
	Let $T$ be a type. Then
\begin{enumerate}
\item[$(1)$] $\mathcal C_T$ is  smooth and of finite type over $\ZZ$ with geometrically connected fibres.
\end{enumerate}
Suppose that $T \neq (2;n)$. Then
\begin{enumerate}
\item[$(2)$] $\mathcal{C}_T$ is separated over $\ZZ$.
\item[$(3)$] $\mathcal C_T$ is Deligne-Mumford over $\ZZ[1/6]$.
\item[$(4)$] There exist a smooth quasi-projective $\ZZ[1/6]$-scheme $U_T$ and an \'etale   surjective morphism $U_T\to \mathcal C_{T,\ZZ[1/6]}$.
\end{enumerate}
\end{proposition}
\begin{proof}
	The scheme $\mathrm{Hilb}_T$ is   smooth  of finite type over $\mathbb Z$ with geometrically connected fibres;  see \cite[Prop.~1.9]{Del72} or \cite[Prop.~2.2.1]{BenoistThesis}.
	Therefore, as $\PGL_{n+c+1}$ is smooth of finite type over $\Spec \ZZ$,
	we see that $\mathcal C_T = [ \PGL_{n+c+1} \backslash \mathrm{Hilb}_T ]$
	is also smooth of finite type over $\ZZ$ with geometrically connected fibres (see \cite[\S 2.3.1]{BenoistThesis}).
	If $T\neq (2;n)$, then $(2)$ and $(3)$ are \cite[Thm.~1.6]{Ben13} and \cite[Thm.~1.7]{Ben13},
	respectively. 
	 Finally, the  last statement  follows from the definition of a smooth finite type  Deligne-Mumford stack
	 \cite[Def.~4.1]{LMB}.
\end{proof}

An explicit description of the functor of points of $\mathcal{C}_T$ can be found in \cite[\S2.3.2]{BenoistThesis}.
For simplicity, we recall this only for points defined over a perfect field $K$.
In which case, the elements $\mathcal{C}_T(K)$ are pairs $(Y,\mathcal{L})$ where $Y$ is a smooth projective variety
over $K$ equipped with an element $\mathcal{L} \in \Pic_{Y/K}(K)$,
such that after a finite extension of $K$ the element $\mathcal{L}$ becomes an ample line bundle
which embeds $Y$ as a smooth complete intersection of type $T$.
Here $\Pic_{Y/K}$ denotes the Picard scheme of $Y$ over $K$ (see \cite[Ch.~8]{BLR});
one has $\Pic_{Y/K}(K) = (\Pic Y_{\bar K})^{\Gal(\bar K/K)}$.

Let us emphasise that $Y$ is not necessarily a complete intersection over $K$.
The obstruction to $(Y,\mathcal{L})$ being isomorphic	to some complete intersection of type $T$ over $K$ 
lies in the Brauer group $\Br K$ of $K$.
Namely, the Hochschild-Serre spectral sequence (see \cite[Ch.~8, p.~203]{BLR})) yields an exact sequence
$$0 \to \Pic Y \to (\Pic Y_{\bar K})^{\Gal(\bar K/K)} \to \Br K.$$
Hence $(Y,\mathcal{L})$ is a complete intersection of type $T$ if and only if the image of $\mathcal{L}$ in $\Br K$ is trivial.  

\subsubsection{Isom-schemes of complete intersections}
We next move onto a study of the automorphism and isomorphism schemes
of complete intersections. 

Let $B$ be a  scheme and let $X$ and $Y$ be smooth complete intersections of the same type $T$ over $B$. By the theory of Hilbert schemes \cite{GrothendieckHilbertSchemes},
the sheaf on the category of $B$-schemes which associates to a $B$-scheme $Z$ the set of $Z$-linear
isomorphisms $X \times_B Z\to Y\times_B Z$ is representable by a  $B$-scheme, which we denote by
$\Isom_B(X, Y)$ (here by a \emph{linear} isomorphism, we mean an isomorphism which is induced by 
an automorphism of the ambient projective space). We denote by $\Lin_B X = \Isom_B(X,X)$ the group scheme of $B$-linear automorphisms of $X$.
We shall often omit the subscript if $B$ is clear from the context.

 \begin{lemma}\label{lem: isomscheme} 
	Let $B$ be a scheme and let $T \neq (2;n)$ be a type.  If $X$ and $Y$ are smooth complete
intersections of type $T$ over $B$, then the morphism of schemes \[\Isom_B(X,Y)\to B\]  is finite.  
\end{lemma}
\begin{proof} By Proposition \ref{prop:moduli_stack} the stack $\mathcal C_{T}$ is separated over $\ZZ$. Therefore,  the diagonal \[\Delta :\mathcal C_{T,B} \to \mathcal C_{T,B} \times_{B} \mathcal C_{T,B} \] is   proper \cite[Def.~7.6]{LMB}. 
 There is a Cartesian diagram of stacks \[\xymatrix{\Isom_B(X,Y) \ar[rr] \ar[d] & & B \ar[d] \\ \mathcal C_{T,B} \ar[rr]_{\Delta} & & \mathcal C_{T,B} \times \mathcal C_{T,B}} \] where $B\to \mathcal C_{T,B} \times \mathcal C_{T,B}$ is the moduli map associated to $X$ and $Y$.  
   As proper morphisms
are stable by base-change, we see that $\Isom_B(X,Y)\to B$ is proper. Since we are considering linear isomorphisms, the morphism   $\Isom_B(X,Y)\to B$ is affine. As proper affine morphisms of  schemes are finite \cite[Lem.~3.3.17]{Liu2},  this concludes the proof.
\end{proof}

\begin{remark} \label{rem:aut}
	In this paper, we only consider \emph{linear} isomorphisms
	between complete intersections.
	This is not a serious restriction in general;
	the only smooth complete intersections which can admit a non-linear
	automorphism are curves and  K$3$ surfaces (see \cite[Thm.~3.1]{Ben13}).
\end{remark}

\subsection{Hodge theory of complete intersections}  \label{sec: Hodge theory}
In this section we explain the necessary Hodge theory of complete intersections over $\CC$ required in this paper. For basic Hodge theory see \cite{Mullerstach} and \cite{Voi07}. Results particular to  complete intersections can be found in \cite{Del72} and \cite{Rap72}.

Let $X$ be a smooth complete intersection of dimension
$n$ over $\CC$. For each $0 \leq i \leq 2n$ with $i \neq n$, we have
$$\mathrm H^i(X,\CC)= \left\{
\begin{array}{ll}
	0,& \text{if } i \text{ is odd},\\
	\CC,& \text{if } i \text{ is even}.
\end{array}
\right.$$
In particular, the only ``interesting'' cohomology group of $X$ is $\mathrm H^n(X,\CC)$.
By Hodge theory, there is a decomposition $\mathrm H^n(X,\CC)=\bigoplus_{p+q = n} \mathrm H^{p,q}(X)$.
The \emph{Hodge level} of $X$ is defined to be the supremum
$$\ell(X) = \sup\{| p-q |: p+q = n, \quad \mathrm H^{p,q}(X) \neq 0\}.$$
If the above set is empty, then by convention we define the Hodge level to be $-\infty$.
The Hodge level of $X$ is then a well-defined element
of  $\ZZ_{\geq 0}\cup \{-\infty\}$, and should be thought of as a rough measure
of the geometrical complexity of $X$.
Note that 
$\ell(X)~\equiv~n~\bmod 2,$
where by convention $-\infty \equiv 1 \bmod 2$.

Two smooth complete intersections of the same type $T$  over $\CC$ have the same Hodge level, as the Hilbert scheme $\mathrm{Hilb}_T$ is irreducible. Hence we define 
the \emph{Hodge level of $T$} to be the Hodge level of some (hence any) smooth complete intersection  $X$  of type $T$ over $\CC$.
This allows us to define the Hodge level of a smooth complete intersection over  any field $K$  to be the Hodge level of its type.  

The group $\mathrm H^n(X,\ZZ)$ is torsion free, which we view as a sublattice of $\mathrm H^n(X,\CC)$.
The primitive cohomology   $\mathrm{H}^n_{\prim}(X,\CC)$ of $X$ is defined as the kernel of the Lefschetz operator $L:\mathrm{H}^n(X,\CC)\to \mathrm{H}^{n+2}(X,\CC)$ (see \cite[\S6.2.3]{Voi07}). The lattice $\mathrm{H}^n_{\prim}(X,\ZZ)$ is defined to be $\mathrm{H}^n(X,\ZZ)\cap \mathrm{H}^n_{\prim}(X,\CC)$. If $n$ is odd then $\mathrm{H}^n_{\prim}(X,\ZZ) = \mathrm{H}^n(X,\ZZ)$, otherwise $\rank \mathrm{H}^n_{\prim}(X,\ZZ) = \rank \mathrm{H}^n(X,\ZZ) -1$.

\subsubsection{Classification} \label{sec:classification}
The classification of smooth complete intersections of Hodge level at most $1$
was performed by Deligne and Rapoport \cite[\S 2]{Rap72}.
We record this classification here, together with the type
and the $n$th Betti number $b_n$ in relevant cases.

\medskip

\noindent
Hodge level $-\infty$:
\begin{itemize}
	\item Quadric hypersurfaces of odd dimension, $(2;n)$, $n$ odd, $b_n = 0$.
\end{itemize}
Hodge level $0$:
\begin{itemize}
	\item Quadric hypersurfaces of even dimension, $(2;n)$, $n$ even, $b_n=1$.
	\item Cubic surfaces, $(3;2)$, $b_2 = 7$.
	\item Even-dimensional intersections of two quadrics, $(2,2;n)$, $n$ even, \\$b_n = n+4$
\end{itemize}
Hodge level $1$:
\begin{itemize}
	\item One-dimensional non-rational smooth complete intersections.
	\item Cubic threefolds, $(3;3)$, $b_3=10$.
	\item Intersections of a quadric and cubic in $\PP^5$, $(2,3;3)$, $b_3=40$.
	\item Cubic fivefolds, $(3;5)$, $b_5=42$.
	\item Quartic threefolds, $(4;3)$, $b_3=60$.
	\item Odd-dimensional intersections of two quadrics, $(2,2;n)$, $n$ odd,\\ $b_n=n+1$.
	\item Odd-dimensional intersections of three quadrics, $(2,2,2;n)$, $n$ odd,\\
	$b_n=n^2 + 5n +4$
\end{itemize}
In the sequel, for a type $T$ of Hodge level $1$, we let $g(T) = b_n(T)/2$.
We shall freely use that the types which give rise to curves of genus $1$ are $(3;1)$ and $(2,2;1)$, and that, for 
each $g \geq 2$, there are only finitely many types which give rise to smooth curves of genus $g$.

\subsection{Intermediate Jacobians} \label{sec:int_jac}
Let $X$ be a smooth complete intersection of odd dimension $n$ over $\CC$.
The \textit{(analytic) intermediate Jacobian}  $J(X)$ 
of $X$ is defined to be the manifold
\[J(X) =  \mathrm H^n(X,\RR)/\mathrm H^n(X,\ZZ),\]
equipped with its natural structure as a complex torus (see \cite[\S12.1.1]{Voi07}).
This should be thought of as an analogue of the Jacobian of a curve
(unlike the case of curves however, this is just a complex torus in general).
Crucial to this paper is that if $X$ has Hodge level $1$,
then $J(X)$ is in fact a principally polarised abelian variety of dimension
$b_n(X)/2$. The polarisation on $J(X)$ is
induced by the cup-product on $\mathrm H^n(X,\ZZ)$ (see \cite{Del72}). 

Let $T$ be a type of Hodge level $1$. Then Deligne \cite[\S2]{Del72} has constructed a certain principally
polarised abelian scheme over $\Hilb_{T,\QQ}$ (see Section \ref{section: moduli stack} for notation).
In particular, this gives rise to a morphism $\Hilb_{T,\QQ} \to \mathcal{A}_{g(T),1, \QQ}$ of stacks,
where $g(T) = b_n(T)/2$.
It follows from \cite[Lem.~2.11]{Del72} that this morphism is $\PGL_{n+c+1}$-invariant, hence descends
to a morphism of stacks
\begin{equation} \label{eqn:int_jac_stacks}
	J: \quad \mathcal{C}_{T,\QQ} \to \mathcal{A}_{g(T),1, \QQ}
\end{equation}
over $\QQ$. Thus, given a smooth complete intersection $X$ of type $T$ over a field $K$ of characteristic $0$,
we may associate to $X$ a principally polarised abelian variety $J(X)$ over $K$, the \emph{(algebraic) intermediate Jacobian} of $X$. As it arises from a morphism of stacks, this construction is functorial and respects base-change.
We record the cohomological properties of this construction here.

\begin{lemma} \label{lem: int Jac} Let $n$ be an odd integer, let $m=(n-1)/2$ and let $X$ be a smooth $n$-dimensional complete intersection of Hodge level $1$ over a field $K$ of characteristic zero.
	\begin{enumerate}
		\item If $K = \CC$, then the analytification of $J(X)$ is canonically isomorphic to the (analytic)  intermediate
		Jacobian of $X$.		
		\item  If $K=\CC$, then there is a canonical isomorphism \[\mathrm H^1(J(X), \ZZ) \cong \mathrm{H}^n(X, \ZZ(m))\] of polarised $\ZZ$-Hodge structures.
		\item Let $K\to \bar K$ be an algebraic closure and $\ell$ a prime number. Then there is a canonical isomorphism
		$$\mathrm H^n(X_{\bar K},\ZZ_\ell(m)) \cong \mathrm H^1(J(X)_{\bar K}, \ZZ_\ell),$$
		of $\mathrm{Gal}(\bar K/K)$-representations.
	\end{enumerate}
\end{lemma}
\begin{proof}
	This follows from  the results of \cite[\S 2]{Del72} (e.g~\cite[Thm.~2.12]{Del72}).
\end{proof}

The separatedness of $\mathcal C_T$ implies that the intermediate Jacobian defines a separated morphism of stacks.
\begin{lemma}\label{lem:separatedness_of_J}
	Let $T$ be a type of Hodge level $1$. Then the morphism of stacks
	$$J: \quad \mathcal{C}_{T,\QQ} \to \mathcal{A}_{g(T),1, \QQ}$$
	is separated.
\end{lemma}
\begin{proof} The stack
$\mathcal C_{T,\QQ}$ is separated (Proposition \ref{prop:moduli_stack}.(2)), and $\mathcal A_{g(T),1,\QQ}$  has separated diagonal. Hence the separatedness of $J$ follows from \cite[Tag 050M]{stacks-project}. %Lem.~79.4.12.(3)
\end{proof}

\begin{remark}  Deligne has conjectured \cite[\S 3.3]{Del72} that  the intermediate Jacobian of a smooth complete intersection of Hodge level $1$ may also 
	be defined over any field of positive characteristic. We will not require such constructions in this paper (see however Lemma \ref{lem: int Jac has GR}).
\end{remark}

\subsection{A Torelli theorem} \label{Sec:Torelli}
In their famous paper \cite{CG72}, Clemens and Griffiths
proved the global Torelli theorem for cubic threefolds. This work started a flurry
of activity, with other Torelli theorems now known for odd-dimensional intersections of two quadrics
\cite[Cor.~3.4]{Don80} and odd-dimensional intersections of three quadrics \cite[Cor.~4.5]{Deb89}. 
The state-of-the-art concerning global Torelli theorems for cubic fivefolds, quartic threefolds
 and intersections of a quadric with a cubic is however currently much less advanced.

Nowadays, there are many different types of Torelli theorems in the literature (see e.g.~\cite{Catanese}). The main result in this section
is what we have dubbed a ``quasi-finite Torelli theorem'',
which states that there are only finitely many smooth complete intersections (excluding
some special cases) with a given polarised $\ZZ$-Hodge structure. Szendr{\H{o}}i has proved a similar finiteness statement for Calabi-Yau threefolds \cite[Thm.~4.2]{Szendroi}. 
  
\begin{theorem}[Quasi-finite Torelli theorem] \label{thm:qf_torelli}
	Let $\mathrm H$ be a polarised $\ZZ$-Hodge structure and $T$ a type.
	Assume that $T \neq (3;2)$ nor $(2,2;n)$ if $n$ is even.
	Then the set of linear isomorphism classes of smooth complete
	intersections $X$ of type $T$ over $\CC$ for which there is  an isomorphism of polarised $\ZZ$-Hodge structures
	$$\mathrm H^n_{\prim}(X,\ZZ) \cong \mathrm H,$$
 	is finite. 
\end{theorem}
\begin{proof}
 We may  assume that $T \neq (2;n)$ (as all smooth quadrics of the same dimension over $\CC$ are isomorphic).
By Proposition \ref{prop:moduli_stack}, there exist a smooth quasi-projective variety $U$ over $\CC$ and a polarised family $f:Y\to U$ of smooth complete intersections of type $T$ over $U$ such that the induced  moduli map $U\to \mathcal C_{T,\CC}$ is \'etale and, for all smooth complete intersections $X$ of type $T$ over $\CC$, the set \[\{u\in U(\CC) \ : \ Y_{u} \cong X\}\] is non-empty and finite. Explicitly, we choose $U_{T}$ over $\ZZ[1/6]$ as in Proposition \ref{prop:moduli_stack}, take $U$ to be $U_{T,\CC}$ and let $f:Y\to U$ be the pull-back of the universal family over the stack $\mathcal C_{T,\CC}$. Let  $p:U^{\an}\to \Gamma\backslash D$ be the period map associated to
	$f$ and to some base-point $u \in U(\CC)$, where $D$ is the  period domain defined by the polarised $\ZZ$-Hodge structure $\mathrm{H}^n_{\prim}(Y_u, \ZZ)$ and $\Gamma$ is the monodromy group of the polarised family of complete intersections $f:Y\to U$; see \cite[Sect.~4.3-4.4]{Mullerstach} or \cite[Ch.~10]{Voi07} for a detailed treatment of the construction of $D$ and $p$. To prove the theorem, it suffices to show that the fibres of $p$ are finite.

By Selberg's lemma (see  \cite[Thm.~II]{Cassels} or \cite[Lem.~8]{Selberg}), replacing $U$ by an \'etale covering if necessary, we may  assume that $\Gamma$ acts freely on $D$.

Since $\Gamma$ acts freely on $D$, by the infinitesimal Torelli theorem for smooth complete intersections of type $T$ \cite[Thm.~3.1]{Fl86},  the period map $U^{\an}\to \Gamma\backslash D$ is an immersion of complex analytic spaces (here we use  that $T \neq (3;2), (2,2;n)$ with $n$ even, and that $U\to \mathcal C_{T,\CC}$ is \'etale).  By a theorem of Griffiths (see \cite[p.~122]{G}   or \cite[Cor.~13.4.6]{Mullerstach}), there exist a smooth quasi-projective  variety $U^\prime$ over $\CC$, an open immersion $U\to U^\prime$ of schemes and a proper morphism of complex analytic spaces $p^ \prime: U^{\prime,\an}\to D/\Gamma$ extending the period map $p:U^{\an}\to D/\Gamma$. By Stein factorization for proper morphisms of complex analytic spaces \cite[Ch.~10.6.1]{GrauertRemmert}, there exist a proper surjective morphism of complex analytic spaces $p_0:U^{\prime,\an}\to U_0$ with connected fibres, and a finite morphism of complex analytic spaces $U_0\to D/\Gamma$ such that $p^\prime$ factorises as \[\xymatrix{ U^{\prime,\an} \ar[rr]^{ p^\prime} \ar[dr]_{p_0} & & D/\Gamma \\ & U_0. \ar[ur] & } \]  
As the restriction of $p^\prime$ to $U^{\an}$ is an immersion, we conclude that $p_0$ is an isomorphism
when restricted to $U^{\an}$. In particular, the period map $p$ factors as $U^{\an} \subset U_0 \to  D/\Gamma$,
hence has finite fibres.
\end{proof}

We now specialise this result to the case of complete intersections
of Hodge level $1$, where we have the theory of the intermediate Jacobian (see \S \ref{sec:int_jac}). 
\begin{corollary}\label{cor:Jacobian torelli}
	Let $T$ be a type of Hodge level $1$, let $K$ be an algebraically closed field of characteristic $0$ and let
	$A$ be a principally polarised abelian variety over $K$. 
	Then there are only finitely many $K$-isomorphism classes of
	smooth complete intersections over $K$ of type $T$,
	whose intermediate Jacobian is isomorphic to $A$ as a principally polarised abelian
	variety.
\end{corollary}
\begin{proof} 
	By a standard Leftschetz principle type argument, it suffices to prove the result when $K=\CC$.
	Let $\mathrm H = \mathrm{H}^1(A,\ZZ)$ viewed as a polarised $\ZZ$-Hodge structure. If $J(X)$ is isomorphic to $A$, then Lemma \ref{lem: int Jac}
	implies that there is an isomorphism of polarised $\ZZ$-Hodge structures
	\[\mathrm H^n(X,\ZZ(m)) \cong \mathrm H^1(J(X),\ZZ)\cong \mathrm H,\] where $m= (n -1)/2$. Therefore, by the well-known correspondence between polarised $\ZZ$-Hodge structures of type $(-1,0) +(0,-1)$ and polarised abelian varieties, the result follows from Theorem \ref{thm:qf_torelli} on noting that 
	\[\mathrm H(-m) \cong \mathrm H^n(X,\ZZ) = \mathrm{H}^n_{\prim}(X,\ZZ).\qedhere\]
\end{proof}

\begin{remark}
	Note that Theorem \ref{thm:qf_torelli} is ``sharp'', as
	any two cubic surfaces or any two intersections of two quadrics of the same
	even dimension over $\CC$ have isomorphic Hodge structures, respectively.
\end{remark}

%\begin{remark} Let $H$ be a polarized $\ZZ$-Hodge structure and let $h$ be a polynomial of degree $n$. A   smooth projective connected variety $X$ over $\CC$ is  \emph{Calabi-Yau} if the canonical bundle $\omega_X$ is trivial and  $\mathrm{H}^i(X,\mathcal O_X) = 0$ for all $0< i < \dim X$. Note that the techniques used in the proof of Theorem \ref{thm:qf_torelli} can be used to show the following well-known generalization of Szendroi's finiteness result \cite[Thm.~4.2]{Szendroi}: the set of isomorphism classes of  polarized    Calabi-Yau varieties $(X,L)$ with Hilbert polynomial $h$ and Hodge structure $\mathrm{H}^{n}_{\prim}(X,\ZZ) \cong H$ is finite. (Indeed, Calabi-Yau varieties have unobstructed deformations by Tian-Todorov (see \cite{Kawamata}), satisfy the infinitesimal Torelli theorem \cite[Ex. 5.6.2]{Mullerstach} and the period map can be compactified \cite[Cor.~13.4.6]{Mullerstach}.)
%\end{remark}

\subsection{Induced automorphisms on cohomology}\label{sec:automorphisms_faithful}
 The aim of this section is to show that the automorphism group of a smooth complete intersection
acts faithfully on its cohomology, under suitable conditions. Analogues
of this result are known for many other classes of varieties; for example
for curves of genus at least two, abelian varieties \cite[Thm.~3, p.~176]{MumAb}, K3 surfaces  \cite[Prop.~VIII.11.3]{BHPV},
varieties with very ample canonical bundle \cite[p.~37]{PoppI}, certain hyperk\"ahler manifolds \cite[Prop.~10]{Beauville2} 
and certain surfaces \cite{Peters}.

We begin with a lemma on actions of inertia groups of stacks on tangent spaces.  For the definition of the inertia group $I_x$ of a geometric point $x$ of an algebraic stack $\mathcal C$, see \cite[Tag 036X]{stacks-project}, \cite[Tag 050P]{stacks-project} or \cite[\S 2.1]{NoohiIn}. The following lemma is a minor variant of \cite[Prop.~4.4]{EHKV} and is presumably known to the authors of \emph{loc.~cit.}; we give a proof for completeness.

\begin{lemma}\label{lem:hartmann}
	Let $k$ be a field of characteristic $0$. Let $\mathcal C$ be a smooth irreducible finite type separated Deligne-Mumford stack over $k$ whose generic inertia group is trivial. Let $x \in \mathcal C$ be a geometric point  with inertia group $I_x$. Then $I_x$ acts faithfully on the tangent space of $\mathcal C$ at $x$.
\end{lemma}
\begin{proof}
	We may assume that $k$ is algebraically closed.
	By \cite[Prop.~4.4]{EHKV}, the group $I_x$ acts faithfully on the $r$th jet space $J^r_x$ of $\mathcal{C}$ at $x$, for some $r\geq 1$.
	To prove the lemma, it suffices to show that one may take $r=1$ in \emph{loc.~cit}.~when $\mathrm{char}(k) = 0$ (this does not hold
	  in positive characteristic; see \cite[Ex.~4.5]{EHKV}).

By \cite[Thm.~2.12]{Olsson} there exist a smooth affine scheme $U$ over $k$ endowed with an action of $I_x$ and a representable \'etale  morphism $\iota_x: [U/I_x]\to \mathcal C$ such that the image of $\iota_x$ contains $x$. The morphism $\iota_x$ induces  natural $I_x$-equivariant isomorphisms of jet spaces. Therefore, to prove the lemma, we may assume that $\mathcal C = [U/I_x]$.

	The morphism $U\to \mathcal C$ is representable and \'etale, thus induces $I_x$-equivariant isomorphisms on jet spaces. 
	So let $P$ be a point of $U$ lying over $x$, let  $\mathfrak{m}$ be the maximal ideal of $\mathcal O_{U,P}$, and let $\sigma \in I_x$.
	As $\mathrm{char}(k) = 0$ and $\mathcal O_{U,P}$ is regular, by \cite[Lem.~7.1]{Hartmann} and \cite[Rem.~7.2]{Hartmann} 
 	there exists a system of uniformising parameters $x_1,\ldots,x_s \in \mathfrak{m}$ and roots of unity $\zeta_1,\ldots,\zeta_s$ such that $\sigma(x_i) = \zeta_i x_i$.
 	As $\sigma$ acts non-trivially on $J^r_x$, it acts non-trivially on $\mathfrak{m}/\mathfrak{m}^{r+1}$. In particular, there is some $i\in \{1,\ldots,s\}$ for which $\zeta_i \neq 1$. It follows that $\sigma$ acts non-trivially
	on $\mathfrak{m}/\mathfrak{m}^2 = \langle \overline{x_1},\ldots, \overline{x_s} \rangle$, as required.  
\end{proof}

We next obtain a simple criterion for the automorphism group of a smooth complete intersection to act faithfully on its cohomology. To prove this, we use Lemma \ref{lem:hartmann} and Flenner's 
infinitesimal Torelli theorem (as used already in Section \ref{Sec:Torelli}).

\begin{proposition}\label{prop:crit_for_rep}
Let $T$ be a type such that there exists a smooth complete intersection of type $T$ over $\CC$ with no non-trivial linear automorphisms. Then for all smooth complete intersections $X$ of type $T$ over $\CC$, the homomorphism \[\Lin(X) \to \Aut( \mathrm{H}^n(X,\CC))\] is injective.
\end{proposition}
\begin{proof}
	Note that the hypothesis cannot hold if $T = (2;n),(3;1)$ or $(2,2;n)$,
	as such $X$ always have a non-trivial linear automorphism group (in the latter case this  	
	follows from the fact that we can simultaneously diagonalise both quadrics, 
	see e.g.~the proof of Proposition \ref{prop:two_quadrics}). 
	Moreover, the result is well-known when $T=(3;2)$ (see e.g.~\cite[Prop.~8.2.31]{Dol12}).
	We may therefore assume that $T$ is none of these types.

By Flenner's infinitesimal Torelli theorem \cite[Thm.~3.1]{Fl86}, for a smooth complete intersection $X$ of type $T$ the natural
$\Lin(X)$-equivariant homomorphism
\begin{align*}
\mathrm{H}^1(X,\Theta_X) & \to  \Hom(\mathrm{H}^{p,q}(X), \mathrm{H}^{p-1,q+1}(X))
\end{align*}
is injective for some $p,q$, where $\Theta_X$ denotes the tangent bundle of $X$. Therefore to prove the proposition, it suffices to show that the homomorphism
 \[\Lin(X) \to \Aut( \mathrm{H}^1(X,\Theta_X))\] is injective.
 
To do this, recall that $\Lin(X)$ is the inertia group of $\mathcal{C}_{T,\CC}$ at $X$, and that, by deformation theory, the tangent space to $\mathcal{C}_{T,\CC}$ at $X$ is some vector subspace $V \subset \mathrm{H}^1(X,\Theta_X)$ which is stable under the action of $\Lin(X)$. Under our assumptions on $T$, the stack $\mathcal{C}_{T,\CC}$ is a smooth irreducible finite type separated Deligne-Mumford stack (Proposition \ref{prop:moduli_stack}) with trivial generic inertia group. Hence Lemma \ref{lem:hartmann} implies that $\Lin(X)$ acts faithfully on $V \subset \mathrm{H}^1(X,\Theta_X)$, as required.
\end{proof}

In order to apply Proposition \ref{prop:crit_for_rep}, we now show that the assumptions hold in some special cases.

\begin{lemma}\label{lem:d1d2}
Let $T=(d_1,\ldots,d_c;n)$ be a type with $3\leq d_1<d_2\leq \cdots \leq d_c$ and $T \neq (3;1)$. Then the  general smooth complete intersection of type $T$ over $\CC$ has no non-trivial linear automorphisms.
\end{lemma}
\begin{proof} 
	In the case of hypersurfaces the result is well-known, see e.g.~\cite[Thm.~1.5]{Poo05} or \cite{MM63}.
	So assume that $c \geq 2$ and  let 
	$Z$ be a smooth hypersurface of type $(d_1;n+1)$ with no non-trivial linear automorphisms. 
	If $Y$ is a general smooth complete
	intersection of type $(d_2,\ldots,d_c;n+1)$, then $X= Z \cap Y$ is a smooth complete
	intersection of type $T$. 
	
	We claim that $\Lin X$ is trivial.
	Indeed, let $\sigma \in \Lin X$. By our assumption on $T$, the vector space $H^0(X, \mathcal{I}_X(d_1))$
	is one-dimensional and is generated by the equation defining $Z$,
	where $\mathcal{I}_X$ denotes the ideal sheaf of $X$.
	Hence $\sigma$ induces an automorphism of $Z$. 
	However by construction $\Lin Z$ is trivial, thus this induced automorphism is trivial, as required.
\end{proof}

\begin{lemma}\label{lem:222n}
Let $T= (2,2,2;n)$ and $n\in \NN$. Then the general smooth complete intersection of type $T$ over $\CC$ has no non-trivial automorphisms.
\end{lemma}
\begin{proof}
	We prove the result using work of Beauville \cite{Bea77}.
	Let $X$ be a smooth complete intersection of type $(2,2,2;n)$ over $\CC$ and
	let $C$ be the discriminant curve of the associated net of quadrics (see \cite[\S6]{Bea77}).
	If $X$ is chosen generically, then $C$ will be an irreducible plane 
	curve of degree $n+4$.	In which case, \cite[Prop.~6.19]{Bea77} implies that the natural map
	$\Lin X\to\Aut C$
	is injective.
	However, by \cite[Prop.~6.23]{Bea77} (see also \cite[\S4.1.3]{Dol12}) every smooth irreducible plane curve
	of degree $n+4$ is the discriminant curve for some smooth complete intersection of type $(2,2,2;n)$.
	Hence, choosing such a curve  with trivial automorphism group (using Lemma \ref{lem:d1d2}, say)
	yields the result.
\end{proof}

In the special case of a smooth intersection of a quadric with a cubic in $\PP^5_\CC$, we verify directly that the automorphism group acts faithfully on the cohomology. To do this, note that Lemma~\ref{lem: int Jac} implies
that for any smooth complete intersection $X$ of Hodge level $1$ over $\CC$ we have a commutative diagram 
\begin{equation} \label{diag:J(X)} 	
\begin{split} 	
\xymatrix{ \Lin X  \ar[r] \ar[dr]&	\Aut J(X)  \ar[d] \\ 	& \Aut  \mathrm{H}^n(X,\ZZ(m)). } 	
\end{split}  	
\end{equation}

\begin{lemma}\label{lem:23n}
	Let $X$ be a smooth complete intersection of type $(2,3;3)$ over $\CC$. 
	Then $\Lin X \to \Aut \mathrm{H}^3(X,\ZZ)$ is injective.
\end{lemma}
\begin{proof}
	As the morphism on the right in \eqref{diag:J(X)} is injective, it suffices
	to show that the morphism $\Lin X \to \Aut J(X)$ is injective.
	We do this using the method of Beauville given in \cite{Bea12}.
	Namely, let $X$ be a smooth complete intersection of type $(2,3;3)$ equipped
	with a linear automorphism $\sigma$. Choose a faithful representation $V$ of $\sigma$ which
	realises the action on $X \subset \PP(V)$ and such that $\sigma$ has at least one trivial eigenvalue. 
	As in the proof of \cite[Lemma]{Bea12}, there is a short exact
	sequence
	$$0 \to V \to V \oplus \mathrm{Sym}^2V/\langle Q\rangle \to T_0(J(X)) \to 0$$ which is equivariant with respect to the
	action of $\sigma$. Here $\mathrm{Sym}^2V$ denotes the symmetric square of $V$, $Q$ is a choice of quadratic form
	which vanishes on $X$ and $T_0(J(X))$ is the tangent space of $J(X)$ at the origin.
	To prove the result, it suffices to show that $T_0(J(X))$ is not the trivial representation.
	To do this we may consider the associated character, which, as characters are additive on short exact sequences,
	is non-trivial if $\sigma$ acts on $\mathrm{Sym}^2V$ with at least $2$ non-trivial eigenvalues. However,
	this easily follows from the fact that $V$ is a faithful representation of $\sigma$
	with at least one trivial eigenvalue, as required.
\end{proof}

We now come to the main result of this section.

\begin{proposition} \label{prop:faithful}
Let $T=(d_1,\ldots,d_c;n)$ be a type. Assume that one of the following holds.
	\begin{enumerate}
		\item $T$ is of general type, i.e. the inequality $d_1+\cdots +d_c \geq n + c +2$ holds.
		\item $3\leq d_1<d_2\leq \cdots \leq d_c$ and $T\neq (3;1)$. 
		\item $T$ has Hodge level $1$, $T\neq (3;1)$ and $T \neq (2,2;n)$ with $n$ an odd integer.
	\end{enumerate}
	If $X$ is a smooth complete intersection of type $T$ over $\CC$, then 
	the group of linear automorphisms  $\Lin X$ of $X$ acts faithfully on $\mathrm{H}^n(X,\CC)$.
\end{proposition}
\begin{proof} 
	For part $(1)$, the canonical bundle  
	\[
	  \omega_X= \OO_X(d_1+\cdots +d_c - n - c -1)
	\]
	is very ample and thus 
	the argument of \cite[p.~37]{PoppI} applies. We give a brief sketch of this proof 
	to illustrate why it only applies when $T$ has general type. As $\Lin X$ 
	acts faithfully on $\mathrm{H}^0(X, \OO_X(1))$ it also acts faithfully on $\mathrm{H}^0(X, \omega_X)$.
	However, since $\mathrm{H}^0(X, \omega_X)\subset\mathrm{H}^n(X,\CC)$ by Hodge theory,
	we conclude that $\Lin X$ also acts faithfully on $\mathrm{H}^n(X,\CC)$, as required.
	
	Part $(2)$ follows from Proposition \ref{prop:crit_for_rep} and Lemma \ref{lem:d1d2}.
	For part $(3)$, by the classification (see Section \ref{sec:classification}), the only
	types not covered by part $(2)$ are $(2,2,2;n)$ and $(2,3;3)$. These 
	types are handled by Proposition \ref{prop:crit_for_rep}, Lemma \ref{lem:222n}
	and Lemma~\ref{lem:23n}, respectively.
	This completes the proof.
\end{proof}

\begin{remark} 
	Quadrics and curves of genus $1$ are easily seen to not satisfy
	Proposition \ref{prop:faithful}. Indeed the group scheme of (not necessarily linear) automorphisms has a non-trivial identity
	component in this case, which by continuity must act trivially on the lattice $\mathrm{H}^n(X,\ZZ) \subset \mathrm{H}^n(X,\CC)$.
	This component always contains non-trivial linear automorphisms (e.g.~translation by a $3$-torsion point of the Jacobian
	when $X$ is a cubic curve).
	The exception $(2,2;n)$  with $n$ odd is also genuinely required;
	we give counter-examples in Section \ref{section:intnsof2}.
	
	Other examples of varieties for which the automorphism group does not act
	faithfully on the cohomology   have been studied in the case of surfaces of general type \cite{C3}
	and Enriques surfaces \cite{Mukai}.
%and  hyperk\"ahler manifolds; see \cite[Prop.~9]{Beauville2},  \cite{Boissiere}, \cite{C3}, \cite{MukaiNamikawa} and %\cite{Mukai}.
\end{remark}

\begin{remark}
	We have only proved Proposition \ref{prop:faithful} for the cases which will be required in this paper.
	It is quite likely that Proposition \ref{prop:faithful} holds in greater generality, with the only exceptions
	being curves of genus $1$, quadric hypersurfaces and odd-dimensional complete intersections of two quadrics.
\end{remark}

\section{Arithmetic Torelli theorems} \label{Sec:Arithmetic}
The aim of this section is to prove Theorem \ref{thm:arithmetic_torelli}
and Theorem \ref{thm:global_arithmetic_torelli},
and to also show that the analogues of these results fail for intersections of two quadrics. 
  
\subsection{Twists, torsors and cohomology} \label{sec:twists}
We begin with some remarks on the relationship between twists, torsors and cohomology.

\subsubsection{Torsors and cohomology}
Let $G$ be a smooth affine group scheme over a scheme $B$.
Recall that a faithfully flat finite type $B$-scheme $E$ is a \emph{$G$-torsor} if it is endowed with a left action of $G$ such that the morphism \[E\times_B G \to E\times_{B} E, \quad (x,g)\mapsto (x,x\cdot g)\] is an isomorphism.
Note that, by \cite[Thm.~III.4.3.a)]{MilneEC} and \cite[Prop.~III.4.6]{MilneEC}, the first \v{C}ech cohomology set $\mathrm H^1(B,G)$ with respect  to the \'etale topology classifies $G$-torsors over $B$.

\subsubsection{Twists of complete intersections}
Let now $T$ be a type and suppose that $X$ is a smooth complete intersection
of type $T$ over $B$.

Let $Y$ be a smooth complete intersection of type $T$ over $B$. We say that $Y$ is a \emph{twist} of $X$
if $Y$ is $B$-linearly isomorphic to $X$, locally for the \'{e}tale topology of $B$.
The scheme $Y$ corresponds to some element $[Y]$ of the pointed set $\mathrm H^1(B,\Lin_B X)$;
explicitly $[Y]$ is the class of the $\Lin_B(X)$-torsor $\Isom_B(X,Y)$, with two such twists
having the same class if and only if they are $B$-linearly isomorphic.

Not every element of $\mathrm H^1(B,\Lin_B X)$ is represented by such a twist in general.
It is possible however to give a geometric description of this set, which for simplicity we
only do when $B=\Spec K$ and $K$ is a perfect field. In which case, a simple descent argument, using for example the explicit
description of $C_{T}(K)$ given in Section \ref{section: moduli stack}, shows that $\mathrm{H}^1(K, \Lin_K X)$
classifies those elements of $[C_{T}(K)]$ which become isomorphic to $X$ over $\bar K$.

\subsection{Complete intersections of Hodge level $1$} \label{section:arithmetic_torelli_proof}

We first prove the following stack-theoretic version of the arithmetic Torelli theorem.
\begin{proposition}\label{prop:12}
	Let $T$ be a type of Hodge level $1$ with $T \neq (3;1)$ and $T \neq (2,2;n)$.
	Then the morphism of stacks
	$$J: \quad \mathcal{C}_{T,\QQ} \to \mathcal{A}_{g(T),1, \QQ}$$
	is separated, representable by schemes, unramified, and quasi-finite.
\end{proposition}
\begin{proof} Write $\mathcal C=\mathcal C_{T,\QQ}$ and $\mathcal A = \mathcal A_{g(T),1,\QQ}$.
	The separatedness of $J:\mathcal C\to \mathcal A$ is Lemma \ref{lem:separatedness_of_J}. 
	By Lemma \ref{lem: int Jac}, Proposition \ref{prop:faithful} and \eqref{diag:J(X)},
	for all smooth complete intersections $X$ of type $T$ over $\CC$ the homomorphism $\Lin X\to\Aut J(X)$ is injective. Therefore, by \cite[Tag 04Y5]{stacks-project}, the geometric fibres of $J$ are algebraic spaces. Hence, by \cite[Cor.~2.2.7]{Conrad}, the morphism $J$ is representable by algebraic spaces.

By Corollary \ref{cor:Jacobian torelli} we see that for all schemes $S$ and all morphisms $S\to \mathcal A$, the induced morphism of algebraic spaces $\mathcal C \times_{\mathcal A} S\to S$ is quasi-finite. Thus, since $J$ is separated, it follows from Knutson's criterion  \cite[Cor.~II.6.16]{Knutson} %(see also \cite[Thm. A.2]{LMB} and \cite[Rem. 2.2.6]{Conrad}) 
    that $\mathcal C\times_{\mathcal A} S$ is a scheme. Hence $J$ is representable by schemes and quasi-finite. 
    
    To conclude, note that  Flenner's infinitesimal Torelli theorem \cite[Thm.~3.1]{Fl86} implies that the morphism $J_\CC$ is injective on tangent spaces. As it is representable by schemes, it follows that $J$ is unramified \cite[Tag 0B2G]{stacks-project}.
\end{proof}

\begin{proof}[Proof of Theorem \ref{thm:arithmetic_torelli}]
The classification given in Section \ref{sec:classification} implies that there are only finitely many types with the same Betti numbers. Hence we may assume that the type $T$ is fixed. In which case the result follows from Proposition \ref{prop:12}, which implies that every fibre of the induced map $[C_{T,K}(K)] \to [\mathcal{A}_{g(T),1,K}(K)]$ is finite.
\end{proof}

\subsection{Global arithmetic Torelli}

We now show that in some special cases, we can say even more by combining known global Torelli theorems with the representability
of the morphism of stacks $J:\mathcal C_{T,\QQ}\to \mathcal A_{g(T),1,\QQ}$. 

To state our result in its most general form, recall that a morphism of stacks $f:X\to Y$   is \emph{universally injective} if it is representable by schemes, and for any scheme $S$ and any morphism $S\to Y$, the induced morphism of schemes $X\times_S Y\to Y$ is universally injective \cite[Tag 01S3]{stacks-project}.

\begin{proposition}\label{prop:univ_inj}
 	Let $T = (3;3)$ or $(2,2,2;n)$ with $n$ odd.
	Then the morphism of stacks
	$$J: \quad \mathcal{C}_{T,\QQ} \to \mathcal{A}_{g(T),1, \QQ}$$
	is separated, representable by schemes, unramified, and universally injective.
\end{proposition}
\begin{proof}
	By Proposition \ref{prop:12}, it suffices to show that $J$ is universally injective. To do this, by \cite[Tag 03MU]{stacks-project}, it suffices to show that the non-empty geometric fibres of $J$ are singletons. 
However, for the special types we are considering,
	one knows a global Torelli theorem over $\CC$,
	due to Clemens and Griffiths \cite[(0.11)]{CG72} and Debarre \cite[Cor.~4.5]{Deb89}, respectively. This proves the result.
\end{proof}

\begin{proof}[Proof of Theorem \ref{thm:global_arithmetic_torelli}] The theorem follows easily from Proposition \ref{prop:univ_inj}.
\end{proof}

\subsection{Intersections of two quadrics}\label{section:intnsof2}  
In this section, we show that the analogues of Theorem \ref{thm:arithmetic_torelli}, Proposition \ref{prop:faithful} and Proposition \ref{prop:12} fail for intersections of two quadrics
of odd dimension, so that the hypotheses in these statements are genuinely required.  
That these results fail for curves of genus $1$ is well-known; though it also follows from applying
our results to the special type $(2,2;1)$.
Throughout this section $K$ is a field of characteristic $0$.

\begin{lemma} \label{lem:Wittenberg}
	Let $n$ be odd and $(a_0,\ldots, a_{n+2}) \in K^{n+3}$ be such that 
	\begin{equation} \label{eqn:diagonal}
	X: \quad x_0^2 + \cdots + x_{n+2}^2  = 0, \quad  a_0x_0^2 + \cdots + a_{n+2}x_{n+2}^2  = 0 \quad \subset \PP^{n+2}
	\end{equation}
	is a smooth complete intersection. For $0 \leq i \leq n+2$, let $\sigma_i$ be the automorphism of $X$ given
	by $\sigma_i(x_i) = -x_i$ and $\sigma_i(x_j) = x_j$ for $j \neq i$.
	Then $\sigma_i$ induces	multiplication by $(-1)$ on $J(X)$.
\end{lemma}
\begin{proof}
	This result is proven in \cite{Wit12}.
	We are grateful to Olivier Wittenberg for allowing us to reproduce
	this proof here. 
	
	We may assume that $K=\CC$. Let $m = (n-1)/2$.
	We shall use the explicit description of the intermediate
	Jacobian provided by Reid \cite[Ch.~4]{Rei72} (see also \cite{Don80}).	
	Let $I$ denote the variety parametrising pairs $(p,q)$,
	where $q$ is a quadric in $\PP^{n+2}$ containing $X$ and $p$
	is an $(m+1)$-plane in $q$. Denote by $I \to C \overset{\smash{\pi}}{\to} \PP^1$ the Stein
	factorisation of the projection onto the second coordinate.
	Then $C$ is a hyperelliptic curve of genus $m+1$, whose Jacobian $J(C)$
	is canonically isomorphic to $J(X)$
	(see \cite[Thm.~4.14($c'$)]{Rei72}).
	
	Let now $\sigma = \sigma_i$, for some $i$. Then $\sigma$ leaves invariant
	each quadric containing $X$, hence induces an automorphism $\sigma_C$
	of $C$ which respects $\pi$. To see that $\sigma_C$ is non-trivial,
	let $Q$ be a smooth quadric hypersurface containing $X$. This contains exactly
	two families of $(m+1)$-planes, and it suffices to show that $\sigma$
	permutes these. As in the proof of \cite[Lem.~1.2]{Don80}, one may
	apply induction to reduce to the case where $n=-1$, i.e. where
	$Q$ has the form 
	$$c_0x_0^2 + c_ix_i^2 = 0,$$
	for some $c_0,c_i \in \CC^*$. In which case, the result is clear.
	Therefore $\sigma_C$ is non-trivial and respects $\pi$, hence must be the hyperelliptic involution on $C$.
	However, the hyperelliptic involution acts as multiplication by $(-1)$ on $J(C)$, which proves
	the result.
\end{proof}

We obtain the following, which, on using \eqref{diag:J(X)}, is easily seen to imply that
the analogue of Proposition \ref{prop:faithful} fails in this case.

\begin{proposition} \label{prop:two_quadrics}
	Let $X$ be an odd-dimensional smooth complete intersection of two quadrics over $K$.
	Then the natural morphism of group schemes
	$$\Lin X \to \Aut J(X)$$
	has non-trivial kernel.
\end{proposition}
\begin{proof}
	To prove the result, we may assume that $K=\CC$.
	In which case, we may simultaneously diagonalise both quadrics
	(see \cite[Prop.~2.1]{Rei72}), so that $X$ has the form \eqref{eqn:diagonal}.
	By Lemma \ref{lem:Wittenberg}, the automorphism $\sigma_0 \sigma_1$
	induces the trivial	automorphism of $J(X)$, as required.
\end{proof}

We now show that the analogue of Theorem \ref{thm:arithmetic_torelli}
fails for intersections of two quadrics over suitable fields (this applies to number fields,
or, more generally, to Hilbertian fields).

\begin{proposition}\label{prop:examples_with_same_int_Jac} 
	Let $n$ be odd and assume that $K^*/K^{*2}$ is infinite.
	Then there exist infinitely many non-$K$-linearly isomorphic smooth complete intersections $\{X_i\}_{i \in I}$
	of type $(2,2;n)$ over $K$ such that
	$$J(X_i) \cong J(X_j) \quad \forall \, i,j  \in I,$$
	as principally polarised abelian varieties.
\end{proposition}
\begin{proof}
	We prove the result by constructing explicit counter-examples. 
	Let $(a_0,\ldots, a_{n+2}) \in K^{n+3}$ be such that 
	$$
		X: \quad x_0^2 + \cdots + x_{n+2}^2  = 0, \quad  a_0x_0^2 + \cdots + a_{n+2}x_{n+2}^2  = 0,
	$$
	is a smooth complete intersection. 
	By the functoriality of the intermediate Jacobian, the twists with the same intermediate Jacobian
	as $X$ are classified by the kernel	of the map
	\begin{equation} \label{eqn:kernel_infinite}
		\mathrm{H}^1(K, \Lin X) \to \mathrm{H}^1(K, \Aut J(X))
	\end{equation}
	of pointed sets. We will show that this kernel is infinite.
	
	Let $\sigma_i$ be as in Lemma \ref{lem:Wittenberg} and let $A$ be the subgroup scheme of $\Lin X$
	generated by $\sigma_0\sigma_1$. By Lemma \ref{lem:Wittenberg} we have $A \subset \ker(\Lin X \to \Aut J(X))$,
	hence the image of $\mathrm{H}^1(K,A)$ in $\mathrm{H}^1(K, \Lin X)$ lies inside
	the kernel of \eqref{eqn:kernel_infinite}.
	However $\mathrm{H}^1(K, A) \cong K^*/K^{*2}$ is infinite by assumption, and the fact that its image
	in $\mathrm{H}^1(K, \Lin X)$ is also infinite follows from a standard twisting argument \cite[Cor.~I.5.4.2]{SerreGaloisCohomology}, as $\Lin X$ is finite (Lemma \ref{lem: isomscheme}). 
	The twists of $X$ obtained this way are explicitly given by
	$$
		bx_0^2 + bx_1^2 + x_2^2 + \cdots + x_{n+2}^2 = 0, \, \,
		ba_0x_0^2 + ba_1x_1^2 + a_2x_2^2 + \cdots + a_{n+2}x_{n+2}^2  = 0,
	$$
	for $b \in K^*$. Our proof shows that they all have the same intermediate Jacobian $J(X)$, yet give rise
	to infinitely many $K$-linear isomorphism classes.
\end{proof}

 We now show that the analogue of Proposition \ref{prop:12} 
fails for intersections of two quadrics.

\begin{corollary} \label{cor:not_rep}
	Let $T = (2,2;n)$ with $n$ an odd positive integer.
	Then the morphism of stacks
	$$J: \quad \mathcal{C}_{T,\QQ} \to \mathcal{A}_{g(T),1, \QQ}$$
	is not representable by algebraic spaces.
\end{corollary}
\begin{proof}
	This follows immediately from   Proposition \ref{prop:two_quadrics} and the implication $(3)\implies (1)$ in \cite[Tag 04Y5]{stacks-project}.
\end{proof}

\begin{remark}
	Corollary \ref{cor:not_rep} gives a ``conceptual explanation'' for why the arithmetic Torelli theorem fails
	for odd-dimensional complete intersections of two quadrics. Namely that the intermediate Jacobian,
	viewed as a morphism of stacks, is not representable.
\end{remark}

\section{Good reduction of complete intersections} \label{Sec:Good}
In this section we define the notion of good reduction for complete intersections and study its basic properties.
The main result (Theorem \ref{thm: twists}) states that, under suitable conditions, a complete
intersection admits only finitely many twists with good reduction. This allows one to reduce to showing that
there are only finitely many $\bar K$-isomorphism classes (rather than $K$-isomorphism classes) 
with good reduction over $B$, when considering problems of Shafarevich-type for complete intersections.

\subsection{Preliminary finiteness theorems}\label{section: classical}
We begin by gathering some classical finiteness results.

\begin{definition} \label{def:smooth_reduction}
Let $B$ be an integral scheme with function field $K$ and let $X$ be
 a proper variety over $K$. A \emph{model for $X$ over $B$} is a flat proper $B$-scheme
 $\mathcal X\to B$ together with a choice of isomorphism $\mathcal{X}_K \cong X$.
 We say that
 \begin{enumerate}
\item 	 $X$ has \textit{smooth reduction at a point $v$} if $X$ has a smooth model over the localisation
	$B_v$ of $B$ at $v$.
\item  $X$ has \textit{smooth reduction over $B$} if
	$X$ has smooth reduction at all points of codimension one of $B$.
 \end{enumerate} 
\end{definition}

Let $B$ be an arithmetic scheme with function field $K$ and structure sheaf $\mathcal O_B$. For example, $K$ is a number field and $B=\Spec \mathcal O_K[S^{-1}]$ with $S$ a finite set of finite places of $K$. 
The first finiteness result we state is a generalisation of the Hermite-Minkowski theorem for number fields to  arithmetic schemes.

\begin{theorem}[Hermite-Minkowski] \label{thm:HM} Let $d$ be an integer. Then there are only finitely many field extensions $L/K$ of degree $d$ such that, for all $v$ in $B$ of codimension one, the field extension $L/K$ is unramified over the discrete valuation ring  $\mathcal O_{B,v}$.
\end{theorem}
\begin{proof}  We may assume that $B$ is affine and smooth over $\ZZ$. 
By a well-known consequence of Hermite's classical finiteness theorem  \cite[p.~209]{FaltingsComplements}
the scheme $B$ has only finitely many finite \'etale covers of degree $d$. 
The result therefore follows from  Zariski-Nagata purity of the branch locus \cite[Cor.~X.3.3]{SGA1}.
\end{proof}

\begin{theorem}[Siegel]\label{thm:units}  
Suppose that $K$ is a number field and let $S$ be a finite set of finite places of $K$.
Then the equation \[x+y = 1\] has only finitely many solutions with $x,y \in \OO_K[S^{-1}]^{*}$.
\end{theorem}
\begin{proof}
	See e.g.~\cite[Thm.~IX.4.1]{Silverman}.
\end{proof}

\begin{theorem}[Faltings] \label{thm: Faltings} Let $g \in \NN$. Then the set of $K$-isomorphism classes of $g$-dimensional principally polarised abelian varieties over $K$ with smooth reduction  over $B$ is finite.
\end{theorem}
\begin{proof} If $\dim B=1$, this theorem is the subject of  \cite{Faltings2}. %  and \cite{Szpiroa}.
 In its full generality, the theorem is proven in  \cite[p.~205, Thm.~2]{FaltingsComplements}.
\end{proof}

The next lemma is a consequence of the theorem of Hermite-Minkowski.

\begin{lemma} \label{lem:HM_cohomology}
	Let $\mathcal{G}$ be a finite  \'{e}tale group scheme over $B$ with generic fibre $G$.
	Then the set
	$$
	\bigcap_{v \in B^{(1)}}\im\left( \mathrm{H}^1(B_v,\mathcal{G}_v) \to \mathrm{H}^1(K,G) \right)
	$$
	is finite.
\end{lemma}
\begin{proof}
	As $G$ is finite, only finitely many elements of $ \mathrm{H}^1(K,G)$ trivialise over any given finite extension of $K$.
	In particular by inflation-restriction (see \cite[I.5.8(a)]{SerreGaloisCohomology}), we may assume that the action of $\Gal(\bar{K}/K)$ on $G(\bar K)$ is trivial. In which case, for $v \in B^{(1)}$ we have
	\begin{align*}
		\mathrm{H}^1(K,G)&= \Hom_{cts}(\Gal(\bar{K}/K), G(K))/ \Inn G(K), \\
		\mathrm{H}^1(B_v,\mathcal{G})&= \Hom_{cts}(\pi_1(B_v), \mathcal{G}(B_v))/\Inn \mathcal{G}(B_v).
	\end{align*} 
	In particular, the elements of  
	$$
	\bigcap_{v \in B^{(1)}}\im\left( \mathrm{H}^1(B_v,\mathcal{G}_v) \to \mathrm{H}^1(K,G) \right)
	$$
	may be represented by certain isomorphism classes of finite field extensions of $K$ of 
	bounded degree which are moreover unramified at all points of codimension one of $B$. 
	Hence the required finiteness follows Theorem \ref{thm:HM}.
\end{proof}

\subsection{Good reduction: Definitions and basic properties}
We now define good reduction and study its basic properties.

\begin{definition} \label{def:good_ceduction_CI}  Let $B$ be an integral scheme with function field $K$.
    Let $T$ be a type and let $X$ be a  complete intersection of type $T$ over $K$. A \textit{good model for $X$ over $B$}
	is a smooth complete intersection $\mathcal X\to \PP^{n+c}_{B}$ of type $T$ over $B$ together
	with a choice of $K$-linear isomorphism $\mathcal{X}_K \cong X$.
	If  $v$ is a point of $B$, a \textit{good model for $X$ at $v$}
	is a good model for $X$ over the localisation $B_v$ of $B$ at $v$. We say that
\begin{enumerate}
\item $X$ has \textit{good reduction at $v$} if $X$ has a good model at $v$. 
\item $X$ has \textit{good reduction over $B$} if
	$X$ has good reduction at all points of codimension one of $B$.
 \end{enumerate}
\end{definition}
We emphasise that having good reduction (in the sense of Definition \ref{def:good_ceduction_CI}) is stronger in general than having smooth reduction
(in the sense of Definition~\ref{def:smooth_reduction}). Good reduction behaves better than smooth reduction, in part due to its relationship with moduli stacks. These points are nicely illustrated by Lemma \ref{lem:glueing}
and Lemma \ref{lem: infinitely many cubics} below. 

We now record a consequence of Lemma \ref{lem: isomscheme} for the unicity of good models. 

\begin{lemma}\label{lem: un} Let $T\neq (2;n)$ be a type and let 
$B$ be an integral Noetherian regular scheme with function field $K$. If $X$ and $Y$ are smooth  complete intersections of type
$T$ over $B$ such that $X_K$ and $Y_K$ are $K$-linearly isomorphic, then $X$ and $Y$ are $B$-linearly isomorphic.
\end{lemma}
\begin{proof}
By Lemma \ref{lem: isomscheme}, the morphism $\Isom_B(X,Y)\to B$ is finite.
Therefore any $K$-rational point of its generic fibre extends
to a section over $B$. This follows from Zariski's main theorem, but also the much stronger statement proven in \cite[Prop.~6.2]{GLL}. 
\end{proof}

The next result shows that having good reduction is closely related to being
an integral point on the moduli stack. This interplay between good reduction
and stacks will occur throughout this paper. 

\begin{lemma}\label{lem:glueing}
	Let $B$ be a Dedekind scheme with function field $K$. Let $T$ be a type and let $X$ be a smooth complete intersection of type $T$ over $K$ which has good reduction over $B$. Then
\begin{enumerate}
\item The $K$-linear isomorphism class of $X$ lies in the 
	image of the map of sets
	$$[\mathcal{C}_T(B)] \to [\mathcal{C}_T(K)].$$
\item If $\Pic(B) = 0$, then $X$ has a good model over $B$.
\end{enumerate}
\end{lemma}
\begin{proof}
	By assumption $X$ admits a good model at each point of $B$.
	We may spread these models out to obtain a Zariski open
	cover $\{B_i\}_{i\in I}$ of $B$ such that $X$ admits a good model 
	$\mathcal{X}_i$ over each $B_i$.
	The generic fibres of each $\mathcal{X}_i$ are pairwise linearly isomorphic,
	hence, refining the cover if necessary, we may glue these to obtain
	a smooth proper model  $h:\mathcal{X} \to B$ of $X$, together with 
	a line bundle $\mathcal{O}_\mathcal{X}(1)$ which is flat over $B$
	and which induces the hyperplane bundle on each fibre.
	It is now easy to see that $(1)$ holds, using for example the explicit
	description of the functor of points of $\mathcal{C}_T$ given in \cite[\S2.3.2]{BenoistThesis}.
	
	For $(2)$, by \cite[Lem.~1.1.8]{BenoistThesis} the line bundle $\mathcal{O}_\mathcal{X}(1)$
	is relatively very ample hence induces an embedding
	$$\mathcal{X} \hookrightarrow \PP(h_*(\mathcal{O}_\mathcal{X}(1))),$$
	which is Zariski locally on $B$ a complete intersection of type $T$.
	The sheaf $h_*(\mathcal{O}_\mathcal{X}(1))$ is locally free on $B$, and since 
	$B$ is a Dedekind scheme with $\Pic B =0$, we find that it is actually free. Therefore
	$\PP(h_*(\mathcal{O}_\mathcal{X}(1))) \cong \PP^{n+c}_B$.
	
	Let $\pi:\PP^{n+c}_B \to B$ denote the structure morphism.
	To complete the proof, it suffices to show that $\mathcal{X} \subset \PP^{n+c}_B$ is a
	complete intersection over $B$. This follows from the fact that 
	kernel of the epimorphism 
	$$h_*(\mathcal{O}_\mathcal{X}(k)) \to \pi_*(\mathcal{O}_{\PP^{n+c}_B}(k))$$
	is locally free, hence free, for all $k \in \ZZ$
	(see the proof of \cite[Prop.~2.1.12]{BenoistThesis}
	or \cite[Prop.~1.9]{Del72}).	
\end{proof}

If one would like finiteness results
of Shafarevich-type (as in Theorem~\ref{theorem: main theorem}) to hold, one needs to use the
``right'' notion for good reduction. Here we present an example to illustrate this point, which is a variant
of an example considered by Scholl \cite[Rem.~4.6]{Scholl}.
Recall that we say that a smooth
cubic surface over a field is \emph{split} if all $27$ lines are defined over that field.

\begin{lemma}\label{lem: infinitely many cubics} Let $B$ be an integral scheme with function field $K$.
	  Any split cubic surface over $K$ has smooth reduction over $B$.
	In particular, if $K$ is infinite, then there are infinitely many pairwise non $\bar K$-isomorphic
	cubic surfaces over $K$ with smooth reduction over $B$.
\end{lemma}
\begin{proof} Let $v \in B^{(1)}$. Any split cubic surface $X$ over $K$ is a blow-up of $\PP^2_K$ in a collection
	of $6$ rational points $P_1,\ldots,P_6$ in general position. These  points uniquely extend
	to $B_v$-points of $\PP^2_{B_v}$.
	Blowing-up  these $B_v$-points successively, we obtain a smooth projective model for $X$ at $v$. Thus, as $v$ was arbitrary, the cubic surface $X$ has smooth reduction over $B$.
	As $K$ is infinite, it is clear from this construction that there are infinitely many
	$\bar K$-isomorphism classes amongst split cubic surfaces. 
\end{proof}
Lemma \ref{lem: infinitely many cubics} shows that in our main result (Theorem \ref{theorem: main theorem}), 
the words ``good reduction'' cannot be replaced by
 ``smooth reduction''. Let us emphasise that Lemma \ref{lem: infinitely many cubics}
 does not contradict Theorem~\ref{theorem: main theorem}, as the fibres of the smooth morphisms constructed in 
the lemma will not all be smooth cubic surfaces in general, but only ``weak'' del Pezzo surfaces. 

\subsection{Twists and good reduction}\label{section: twists}

We now show that a complete intersection admits only finitely many twists with good reduction (see Section \ref{sec:twists}), provided the type is not $(2;n)$.
Our proof of this makes use of Hermite-Minkowski for arithmetic schemes,
and the separatedness of the moduli stack.

\begin{theorem}\label{thm: twists} 
    Let $B$ be an arithmetic scheme with function field $K$ and let $X$ be a smooth complete intersection over $K$ of type $T \neq (2;n)$.
    Then the set of $K$-linear isomorphism classes of complete intersections $Y$ of type $T$ with good reduction over $B$ and which are twists of $X$, is finite.
\end{theorem}
\begin{proof}
 To prove the result, replacing $B$ by a dense open subscheme if necessary,
 we may assume that $X$ has  a good model $\mathcal X\to B$.  
By  Lemma \ref{lem: isomscheme}, since $T \neq (2;n)$, we know that $\Lin_B(\mathcal X)$ is finite over $B$. In particular, 
replacing $B$ again by a dense open subscheme if necessary, we may assume that $\Lin_B(\mathcal X)$ is finite \'{e}tale.

Let $Y$ be a complete intersection of type $T$ which is a twist of $X$ over $K$ and which has good reduction over $B$.
Let $v \in B^{(1)}$, let $\mathcal Y_v$ be a good model for $Y$ over $B_v$
and let $\mathcal X_v = \mathcal X \times_{B} B_v$. 
We claim that $\Isom_{B_v}(\mathcal X_v,\mathcal Y_v)$ is an $\Lin_{B_v}(\mathcal X_v)$-torsor for the \'{e}tale topology.
To prove this, consider the natural left $\Lin_{B_v}(\mathcal X_v)$-action
 $$\Lin_{B_v}(\mathcal X_v) \times_{B_v} \Isom_{B_v}(\mathcal X_v,\mathcal Y_v) \to \Isom_{B_v}(\mathcal X_v,\mathcal Y_v).$$
Let $L/K$ be a finite field extension such that $Y_L$ is isomorphic to $X_L$ over $L$ and let $C_v\to B_v$ be the normalisation of $B_v$ in $L$ (see \cite[Def.~4.1.24]{Liu2}). 
 As  $\Isom_{B_v}(\mathcal X_v,\mathcal Y_v)$ contains an $L$-point,
 it contains a $C_v$-point by Lemma \ref{lem: un}.
 Thus it trivializes over $C_v$, hence is a $B_v$-torsor under $\Lin_{B_v}(\mathcal X_v)$ for the fppf topology.
Since $\Lin_{B_v}(\mathcal X_v)$ is finite \'{e}tale, by fppf descent it is also
a torsor for the \'{e}tale topology, thus proving the claim. 

Hence the class $[Y] \in \mathrm{H}^1(K,\Lin_K(X))$ lies in the image of the natural
map $$  \mathrm{H}^1(B_v, \Lin_{B_v}(\mathcal X_v)) \to \mathrm{H}^1(K, \Lin_K(X)).$$ As $v$
were arbitrary, we conclude that $[Y]$ lies in  
$$ 
\bigcap_{v \in B^{(1)}}\im\left( \mathrm{H}^1(B_v,\Lin_{B_v}(\mathcal X_v)) \to \mathrm{H}^1(K, \Lin_K(X)) \right).$$
The finiteness of this set now follows from Lemma \ref{lem:HM_cohomology}. The result is proved.
\end{proof}

We now give a simple application of Theorem \ref{thm: twists}.
We say that a hypersurface $X$ over $K$ is \textit{geometrically diagonalisable} if $X_{\bar K}$ is $\bar K$-linearly isomorphic to a hypersurface of the form $x_0^d + \cdots + x_{n+1}^d =0$ in $\PP^{n+1}_{\bar K}$.
Examples include hypersurfaces of the shape
$a_0x_0^d + \cdots + a_{n+1}x_{n+1}^d =0$.

\begin{corollary}
	Let $B$ be an arithmetic scheme with function field $K$, let $n\geq 1$ and let $d\geq 3$.
	Then the set of $K$-linear isomorphism classes of geometrically diagonalisable hypersurfaces over $K$
	of dimension $n$ and degree $d$ with good reduction over $B$, is finite.
\end{corollary}
\begin{proof}
	This follows immediately from Theorem \ref{thm: twists}. 
\end{proof}

\section{The Shafarevich conjecture for complete intersections} \label{Sec:Shaf}
In this section we prove Theorem \ref{theorem: main theorem}, together with various
generalisations to arithmetic schemes, by bringing
together the results of the previous sections. As should be clear
from these results, it will be necessary for us to consider
different cases from the Deligne-Rapoport classification separately (see Section~\ref{sec: Hodge theory}),
depending on the different properties of the stack $\mathcal{C}_T$
and the different cases covered by our arithmetic Torelli theorem.

\subsection{Quadrics} \label{sec:quadrics}
We begin with  quadric hypersurfaces.
The result here is a special case of our more general result
on good reduction of flag varieties \cite[Thm.~1.4]{JL2}. For completeness however,
we give a sketch of a proof in order to illustrate the difficulties
arising in the generalisation to arithmetic schemes.

\begin{proposition}\label{prop:quadrics}
	Let $K$ be a number field, let $B \subset \Spec \OO_K$ be a dense open subscheme,
	and let $n \in \NN $.
	Then the set of $K$-isomorphism classes of $n$-dimensional quadric hypersurfaces
	over $K$ with good reduction over $B$ is finite.
\end{proposition}
\begin{proof}
	Let $\mathcal{X}_0 \subset \PP^{n+1}_B$ be a smooth quadric hypersurface over $B$ 
	(this exists for any $B$).
	Let $\PO_{n+1}$ denote the automorphism group scheme of $\mathcal{X}_0$ over $B$.
	
	Let $X$ be a quadric hypersurface of dimension $n$  over $K$ with good reduction over $B$.
	Note that the $K$-isomorphism class of $X$ corresponds to some element $[X] \in \mathrm{H}^1(K, \PO_{n+1}).$
	Moreover $X$ admits a smooth proper model $\mathcal{X} \to B$ whose fibres are isomorphic to smooth quadric
	hypersurfaces. A general result of Demazure (see the remark
	on page $186$ of \cite{Dem77}) implies that $\mathcal{X}$
	is a twist of $\mathcal{X}_0$, hence $[X]$ lies
	in the image of the map $\mathrm{H}^1(B, \PO_{n+1}) \to \mathrm{H}^1(K, \PO_{n+1})$.
	However the cohomology set $\mathrm{H}^1(B, \PO_{n+1})$ is finite by a general
	result of Gille and Moret-Bailly \cite[Prop.~5.1]{GilleMoretBailly}.
\end{proof} 

\begin{remark}
It does not seem to be possible to prove an analogue of Proposition \ref{prop:quadrics} 
over general arithmetic schemes with current tools. The crucial lacking ingredient
is the finiteness of the image $\im(\mathrm{H}^1(B, \PO_{n+1}) \to \mathrm{H}^1(K, \PO_{n+1}))$ when $\dim B >1$.
This is closely related to the finiteness of Tate-Shafarevich sets of linear algebraic groups,
which is not known over finitely generated field extensions of $\QQ$ 
in general. Note that the analogue of the result \cite[Prop.~5.1]{GilleMoretBailly} used in 
Proposition \ref{prop:quadrics} is even false over higher dimensional arithmetic schemes, e.g.~the set $\mathrm{H}^1(\PP^1_\ZZ, \PO_3) 
= \mathrm{H}^1(\PP^1_\ZZ, \PGL_2)$ is infinite, as there are infinitely many non-equivalent $\PP^1$-bundles over $\PP_\ZZ^1$ (Hirzebruch surfaces).
\end{remark}

\subsection{Intersections of two quadrics}

In this section we prove the Shafarevich conjecture for intersections of two quadrics.  To do so, we will use pencils of quadrics; for geometric background, see \cite[\S 22]{Har92}, \cite{Rei72} and  \cite[\S 3.3]{Wit07}.

Let $n$ be a positive integer and let $A$ be an integral domain in which $2$ is invertible. Let
$$X: \quad Q_1(x) = Q_2(x) = 0 \quad \subset \PP^{n+2}_A$$ 
be a smooth complete intersection of two quadrics $Q_1$ and $Q_2$ over $A$. 
Let $$\Delta(X) : \quad \det(\lambda Q_1 + \mu Q_2) = 0 \quad \subset \PP^1_A$$
denote the \textit{discriminant} of the associated pencil of quadrics.
This is a closed subscheme of $\PP^1_A$ of degree $n+3$ which parametrises the degenerate quadrics in the pencil.
Note that here we are committing some (common) abuses of notation. Firstly, in the definition of $\Delta(X)$, we identify  each $Q_i$ with the corresponding symmetric matrix over $A$. Secondly,
the definition of the discriminant depends on the choice of the $Q_i$, however different choices within the same pencil give rise to
$A$-linearly isomorphic subschemes of $\PP^1_A$, which will be sufficient for our purposes. 
Here, as usual, by a linear isomorphism we mean one which is induced by an automorphism of the ambient projective space.

\begin{lemma} \label{lem:pencil} 
	Let $K$ be a field in which $2$ is invertible.  Let $X$ and $Y$ be smooth complete intersections of two quadrics over $K$ with $K$-linearly isomorphic discriminants. Then $Y$ is a twist of $X$ over $K$.
\end{lemma}
\begin{proof} We may assume that $K$ is algebraically closed.
	In which case, the result is well-known; see	\cite[Thm.~22.41]{Har92}.
\end{proof}

\begin{definition}
Let $B$ be an integral scheme with function field $K$. We shall say that a closed subscheme of $\PP^1_B$ that is finite \'{e}tale over $B$ is \emph{split}
if it is $B$-isomorphic to a disjoint union of copies of $B$. A closed subscheme $\Delta$ in $\PP^1_K$ has \emph{good reduction over $B$} if there exists a 
closed subscheme of $\PP^1_B$ that is finite \'{e}tale over $B$ and whose generic fibre is $K$-linearly isomorphic to $\Delta$.
\end{definition}

\begin{lemma}\label{lem:wd}  Let $B$ be an integral affine scheme 
with function field $K$ such that $2$ is invertible in $B$. Let $X$ be a smooth complete intersection of two quadrics over $K$. 
If $X$ has a good model over $B$, then the discriminant of $X$ has good reduction over $B$.
\end{lemma}
\begin{proof} Let $(Q_1,Q_2)$ be a pair of quadrics over $B$ defining a good model  for $X$ over $B$. To prove the lemma, it suffices to show that the polynomial $$\det(\lambda Q_1 + \mu Q_2)$$ is separable over all residue fields of $B$. This follows from \cite[Prop.~3.26]{Wit07}. 
\end{proof}

For $d \in \NN$ we let $M_{d}(B)$ (resp.~$M_{d}(B)^{\textrm{split}}$)  denote the set of $K$-linear isomorphism classes of closed subschemes (resp.~split closed subschemes) of degree $d$ in $\PP^1_K$ that are finite \'etale over $K$ and have good reduction over $B$.

\begin{lemma} \label{lem:split}
	Let $K$ be a number field and let $B \subset \Spec \OO_K$ be a dense open subscheme.
	Then the set $\sqcup_{d=1}^\infty M_d(B)^{\textrm{split}}$ of $K$-linear isomorphism classes 
	of split closed subschemes of $\PP^1_K$ with good reduction over $B$ is finite.
\end{lemma}
\begin{proof}  
	Let $\Delta$ be a split finite \'etale closed subscheme of $\PP^1_K$ with good reduction over $B$, and let $\mathcal X$ 
	be a finite \'etale subscheme of $\PP^1_B$ whose generic fibre is $K$-linearly isomorphic to $\Delta$ 
	(note that $\mathcal{X}$ is also split).
	As $\PGL_2(B)$ acts transitively on triples of disjoint $B$-points of $\PP^1_B$, %Ref: Mile, Algebraic Groups, Lemma 21.31
	we may assume that $\mathcal X$ contains $0,1$ and $\infty$. However,
	a simple application of Theorem \ref{thm:units} shows that the set
	$(\PP^1_B \setminus \{0,1,\infty\})(B)$ is finite, thus
	there are only finitely many choices for $\mathcal X$ up to $B$-linear isomorphism.
	Therefore, there are only finitely many choices for $\Delta$ up to $K$-linear isomorphism, as required.
\end{proof}

We are now ready to prove the Shafarevich conjecture (Conjecture \ref{conj}) for intersections of two quadrics.

\begin{proposition}\label{prop:intersection of two quadrics} 
 Let $K$ be a number field, let $B \subset \Spec \OO_K$ be a dense open subscheme, and let $n \geq 1$.
 Then the set of $K$-linear isomorphism classes of $n$-dimensional  complete intersections of two quadrics over $K$ with good reduction over $B$ is finite.
\end{proposition}
\begin{proof} 
	To prove the proposition, we may assume that $2$ is invertible on $B$ and that $\Pic(B) =0$. By Lemma \ref{lem:glueing}, it suffices to show that the set $Q_{n}(B)$ of $K$-linear isomorphism classes of $n$-dimensional complete intersections of two quadrics over $K$ with a good model over $B$ is finite. To do so,
	note that by Lemma~\ref{lem:wd} the assignment of the discriminant of an intersection of two quadrics
	gives rise to a well-defined map of sets $Q_n(B)\to M_{n+3}(B)$. 
	By Hermite-Minkowski (Theorem \ref{thm:HM}), there exist an integral scheme $B'$ and a finite \'etale morphism $B^\prime\to B$ such that each element of $M_{n+3}(B)$ splits over $B^\prime$. Consider the composed map of sets \[Q_n(B) \longrightarrow M_{n+3}(B) \longrightarrow M_{n+3}(B^\prime)^{\textrm{split}} \]  By Lemma \ref{lem:split}, the set $M_{n+3}(B^\prime)^{\textrm{split}}$ is finite.   On combining Theorem \ref{thm: twists} with Lemma \ref{lem:pencil} ,
	we see that the composed  map has finite fibres. The result is proved. 
\end{proof}
 
\subsection{Complete intersections of Hodge level $1$}
We now handle the case of complete intersections of Hodge level $1$, for which we use the
intermediate Jacobian (see \S \ref{sec:int_jac} for the relevant properties). Here we use the notion of smooth reduction (Definition \ref{def:smooth_reduction}).

\begin{lemma} \label{lem: int Jac has GR}
	Let $B$ be an integral normal Noetherian scheme with function field $K$
	and let $X$ be a complete intersection of Hodge level $1$ over $K$
	with smooth reduction over $B$. Then the intermediate Jacobian $J(X)$
	of $X$ has smooth reduction over $B$.
\end{lemma}
\begin{proof} 
	To prove the result, we may assume that $B$ is a local Dedekind scheme.
	The proof in this case is very similar to the proof
	of \cite[Lem.~3.2]{Del72}, so we shall be brief.
	Let $X$ be an $n$-dimensional complete intersection
	of Hodge level $1$ with smooth reduction over $B$ and let $m=(n-1)/2$.
	As $X$ has smooth reduction over $B$, the inertia group at the closed point $v \in B$
	acts trivially on the $\Gal(\bar{K}/K)$-module 
	$\mathrm{H}^1(X_{\bar{K}},\ZZ_\ell(m))$ for all primes $\ell$ different from the residue
	characteristic of $v$.
	The N\'eron-Ogg-Shafarevich criterion \cite[Thm.~7.4.5]{BLR}
	and Lemma \ref{lem: int Jac} now give the result.
\end{proof}

We now combine this with Faltings's theorem and our arithmetic Torelli theorem to deduce the Shafarevich conjecture
in the remaining cases.

\begin{proposition}\label{prop: HL1}
	Let $B$ be an arithmetic scheme with function field $K$ and $T$ a type of Hodge level $1$.
	Assume that $T \neq (3;1)$ and $T \neq (2,2;n)$. Then the set of $K$-linear isomorphism classes of
	complete intersections of type $T$ over $K$ with smooth reduction over $B$  is finite. 
\end{proposition}
\begin{proof}  Let $X$ be a  complete intersection of  type $T$ over $K$ with smooth reduction over $B$. 
Note that the dimension of the intermediate Jacobian $J(X)$ is 
 determined by $T$.  By Lemma \ref{lem: int Jac has GR},
the abelian variety $J(X)$ has smooth reduction over $B$.
Faltings's finiteness theorem (Theorem \ref{thm: Faltings})
therefore implies that the set of $K$-isomorphism classes of all principally polarised abelian varieties $J(X)$, 
where $X$ runs over all smooth complete intersections of type $T$ with
smooth reduction over $B$, is finite. 
However by Theorem \ref{thm:arithmetic_torelli} (proven in Section \ref{section:arithmetic_torelli_proof}), only finitely many such complete intersections
have isomorphic intermediate Jacobians over $K$, whence the result.
\end{proof}

\begin{remark}
	It is possible to prove a slightly weaker variant of
	Proposition~\ref{prop: HL1}, using Corollary \ref{cor:Jacobian torelli}, Theorem \ref{thm: twists}
	and Lemma \ref{lem: int Jac has GR},
	which avoids the need to appeal to Theorem \ref{thm:arithmetic_torelli}.
	This gives the same finiteness statement over arithmetic schemes 
	but with \emph{smooth reduction} replaced by \emph{good reduction}; it has the advantage 
	however of also working for odd-dimensional intersections of two
	quadrics.
\end{remark}

\subsection{Proof of Theorem \ref{theorem: main theorem}}\label{section: conclusion}
In order to prove Theorem \ref{theorem: main theorem}, we follow the classification (see Section \ref{sec:classification}). For quadrics the result is Proposition \ref{prop:quadrics}.
For cubic surfaces the result follows from Scholl \cite{Scholl}. 
For intersections of two quadrics this is Proposition \ref{prop:intersection of two quadrics}. 
The remaining types with Hodge level $1$, aside from cubic curves, follow from Proposition \ref{prop: HL1},
as by the classification there are only finitely many types with the same dimension and same Betti numbers.

It therefore remains to handle the case of plane cubics. If $X$ is a smooth plane cubic
with good reduction outside of $S$, then the Jacobian $J(X)$ has smooth reduction
outside of $S$ by Lemma \ref{lem: int Jac has GR}. Hence as $X$ runs over all smooth plane cubics,
Faltings's theorem implies that there are only finitely many $K$-isomorphism classes amongst the $J(X)$.
Moreover if $Y$ is another smooth plane cubic with $J(Y) \cong J(X)$, then one easily sees that $Y$
is a twist of $X$ as cubic curves. The result then follows from Theorem \ref{thm: twists}. \qed

\begin{remark}
	Note that the Shafarevich conjecture actually \emph{fails} for the collection of all curves of genus $1$
	(see \cite[p.~241]{Maz86}).
	Nevertheless, Theorem \ref{theorem: main theorem} shows that it holds for the class of smooth plane cubic curves, for example.
\end{remark}

\section{The Lang--Vojta conjecture implies the Shafarevich conjecture}\label{section:Lang_Shaf}

 The aim of this section is to prove Theorem \ref{thm: lang implies shaf}. Our general result
for arithmetic schemes is Theorem \ref{thm: lis}, which gives Theorem \ref{thm: lang implies shaf} as a special case.

\subsection{The Lang--Vojta conjecture}\label{section:lang}
We first recall the Lang--Vojta conjecture on integral points. The original versions of this conjecture appeared in   \cite{Lang2}  and \cite[Conj.~XV.4.3]{CornellSilverman} (see also \cite[\S0.3]{Abr} for a version over arithmetic schemes). Its first striking consequence was obtained by Caporaso-Harris-Mazur \cite{CHM}.

A quasi-projective     scheme  $U$ over a field $K$ of characteristic zero is of \textit{log-general type} 
if for any irreducible component $U'$ of $(U_{\overline{K}})_{\textrm{red}}$, there is a resolution of singularities $V\to U'$ together with a smooth proper variety $X$ and an open immersion of $V$ into $X$
such that $D=X\backslash V$ is a simple normal crossings divisor and $K_X +D $ a big divisor on $X$.  
 
\begin{conjecture} [Lang--Vojta conjecture] \label{lang conj} Let $B$ be  an arithmetic  scheme with function field $K$ and 
let $U$ be a smooth quasi-projective scheme over $B$. If every    subvariety of $U_{K}$ is of log-general type, then   the set $U(B)$  is finite.
\end{conjecture}

Note that the Lang--Vojta conjecture usually states that if $U_K$ has log-general type, then $U(B)$ is not Zariski dense.
By considering the Zariski closure of $U(B)$ in $U$, one easily sees that this implies Conjecture \ref{lang conj}
in the case where  every subvariety of $U_K$ is also of log-general type.

\subsection{A finite \'etale atlas}
A  morphism $U\to \mathcal C$ of   Deligne-Mumford stacks is an \emph{(\'etale) atlas of $\mathcal C$} if $U$ is an algebraic space and $U\to \mathcal C$ is \'etale and surjective.  Note that, as the diagonal of a Deligne-Mumford stack is representable by algebraic spaces, the morphism $U\to \mathcal C$ is representable by algebraic spaces. We emphasise that an atlas $U$ of $\mathcal C$ is not necessarily a scheme.

For example, if $g\geq 1$ and $k\geq 3$, then the moduli stack $\mathcal A_{g,1}^{[k]}$ of principally polarised $g$-dimensional abelian varieties with level $k$ structure  is a \emph{finite} \'etale atlas of  the moduli stack  $\mathcal A_{g,1}$ over $\ZZ[1/k]$. In  \cite[Ch.~VII, Thm.~3.2]{MoretBailly} Moret-Bailly proved that $\mathcal A_{g,1}^{[k]}$ is a quasi-projective scheme (and not merely an algebraic space).  Similarly, the moduli stack of polarised K3 surfaces admits a finite \'etale atlas over some arithmetic curve (see \cite[Thm.~6.1.2]{Rizov}).   

The aim of this section is to prove similar results for complete intersections (see Theorem \ref{thm: finet}).
Our proof uses Proposition \ref{prop:faithful}, which allows us to add level-structure to $\mathcal C_{T,\CC}$ in a naive way.
This approach is based on Popp's paper \cite{PoppI}.

We start with a criterion for an algebraic space over a field to be  quasi-projective. This result is well-known; for lack of reference we include a proof.
\begin{lemma}\label{lem:quotients}
Let $k\subset L$ be a field extension and let $X$ be an algebraic space over $k$. If $X_L$ is a quasi-projective scheme over $L$, then $X$ is a quasi-projective scheme over $k$.
\end{lemma}
\begin{proof} 
	There exists some finitely generated field extension $k_1$ of $k$ over which
	$X$ becomes a quasi-projective scheme. Decomposing the extension $k\subset k_1$ into a tower of 
	extensions, we are reduced to proving the lemma in the two cases:
	\begin{enumerate}
		\item $k \subset L$ is finite.
		\item $L=k(t)$ with $t$ transcendental over $k$.
	\end{enumerate}
	In the first case, the canonical morphism $X_L\to X$ identifies the algebraic space $X$ with the quotient of $X_L$ by a 
	finite locally free equivalence relation. As  explained in \cite[Rem.~V.5.1]{SGA3}, the result then
	follows from \cite[Thm.~V.4.1]{SGA3}. In the second case, there exist a  dense open subscheme $U\subset \mathbb A^{1}_k$ and a quasi-projective scheme $\mathcal X$ over $U$ such that $\mathcal X$ is generically isomorphic to $X_L$. The $U$-algebraic spaces $X_U=X\times_k U$ and $\mathcal X$ are generically isomorphic. Replacing $U$ by a dense open if necessary, by spreading out \cite[Prop.~4.18]{LMB}, we see that $X_U$ is $U$-isomorphic to $\mathcal X$. Hence on choosing some closed point $u \in U$ with residue field $K$, we see that $X_K$ is isomorphic to the quasi-projective $K$-scheme $\mathcal X_u$. We have therefore reduced to case $(1)$, which proves the result.
\end{proof}

We will require results of Griffiths and Zuo on period maps. We gather these in the next lemma.

\begin{lemma}\label{lem:analytic}
Let $f:X\to U$ be a polarised family of smooth projective connected varieties over $\CC$, where $U$ is a smooth quasi-projective variety. If $i\geq 0$ is an integer such that the period map associated to $R^i_{\prim} f_\ast \ZZ$ is an immersion of complex analytic spaces, then  every subvariety of $U$ is of log-general type. Moreover $U^{\an}$ is Brody hyperbolic, i.e.~every holomorphic map $\CC\to U^{\an}$ is constant.
\end{lemma}
\begin{proof}
If $Z$ is a subvariety of $U$ and $Z^\prime\to Z$ is a resolution of singularities, then the period map on $Z^{\prime,\an}$ 
is generically immersive.  In particular, by Zuo's theorem \cite{ZuoNeg}, the variety $Z^{\prime}$ is of 
log-general type (this result has been further generalised in \cite[Thm.~3.3]{Brunebarbe}). Hence $Z$ is of log-general type. Next, as $\CC$ is simply-connected, a result of Griffiths \cite[Thm.~3.1]{VoisinII} implies that any polarisable variation of Hodge structures over $\CC$ is trivial. In particular, as $U^{\an}$ is a complex manifold which admits an immersive period map,
we see that any holomorphic map $\CC\to U^{\an}$ is constant, so that $U^{\an}$ is Brody hyperbolic, as required.
\end{proof}

We now attach level structure to the moduli stack $\mathcal C_{T}$.
\begin{theorem}\label{thm: finet}
Let $T = (d_1,\ldots,d_c;n)$ be a type.   Assume that one of the following holds.
	\begin{enumerate}
		\item $T$ is of general type.
		\item $3\leq d_1 < d_2 = \cdots =d_c$ and $T \neq (3;1)$.
		\item $T$ has Hodge level $1$, $T\neq (3;1)$ and $T \neq (2,2;n)$.
	\end{enumerate} Then 
there exist 
\begin{enumerate}
\item[$(a)$] a number field $K$ and a dense open subscheme $B \subset \Spec \OO_K$,
\item[$(b)$] a  smooth quasi-projective  scheme $U$ over $B$, and
\item[$(c)$]  a finite \'etale $B$-morphism $U\to \mathcal C_{T,B}$ of stacks \end{enumerate}  such 
that any subvariety of $U_K$ is of log-general type, and the complex  manifold $U_{\CC,\sigma}^{\an}$ is Brody hyperbolic with respect to any embedding $\sigma:K\to \CC$.
\end{theorem}
\begin{proof}  
Let $N=n+c+1$.
For all smooth complete intersections $X$ of type $T$ over $\CC$, with $T$ as in the statement of the theorem, the canonical map 
 \[ \Lin(X) \to \Aut(\mathrm H^n(X,\ZZ) ) \] is injective by Proposition \ref{prop:faithful} and the group $\Lin(X)$ is finite by Lemma \ref{lem: isomscheme}. 
Therefore,  we can apply \cite[Prop.~2.17]{PoppI} to see that
 there exist a finite \'etale Galois morphism 
$H_\CC \to \mathrm{Hilb}_{T,\CC}$ of schemes, where   $ \mathrm{Hilb}_T$ is
the Hilbert scheme of smooth complete intersections of type $T$, and a $\PGL_{N}$-action on $H_\CC$ such that
\begin{itemize}
\item the morphism $H_\CC\to \mathrm{Hilb}_{T,\CC}$ is $\PGL_{N}$-equivariant, and
\item the action of $\PGL_{N}$ on $H_\CC$ is proper and without fixed points.
\end{itemize}  

By spreading out and  specialising, we see that there exist a number field $K$, a dense open subscheme $B$ of $\Spec \OO_K$, a finite \'etale Galois morphism $H\to \mathrm{Hilb}_{T, B}$, and a $\PGL_N$-action on $H$ such that the morphism $H\to \mathrm{Hilb}_{T, B}$ is $\PGL_N$-equivariant and the action of $\PGL_N$ on $H$ is proper and without fixed points.  

%As fundamental groups of finite type connected schemes over $\CC$ are finitely generated  \cite[Exp.~II, Thm.~2.3.1]{SGA7I},a simple argument shows that the finite \'etale cover $H_ CC\to \mathrm{Hilb}_{T,\CC}$ can be defined over $\Qbar$. We now use ``spreading out'' to see that there exist a number field $K$, a dense open subscheme $B \subset \Spec \OO_K$, a finite \'etale Galois morphism $H\to \mathrm{Hilb}_{T,B}$ of schemes over $B$ and a $\PGL_{N}$-action on $H$ such that the morphism $H\to \mathrm{Hilb}_{T,B}$ is $\PGL_{N}$-equivariant and the action of $\PGL_{N}$ on $H$ is proper and without fixed points. 

Consider the induced finite \'etale morphism
$ U\to \mathcal C_{T,B}$ 
where $U = [\mathrm{PGL}_{N}\backslash H]$ is the quotient stack.  
 As $U$ has trivial stabilisers,  the stack $ U$ is a smooth   algebraic space over $B$ by \cite[Thm.~2.2.5.(1)]{Conrad} (see also \cite[Cor.~8.1.1]{LMB}). 

Replacing $B$ by a dense open if necessary, Proposition \ref{prop:moduli_stack} implies that $\mathcal C_{T,B}$ is separated 
and Deligne-Mumford over $B$. Let $\mathcal C_{T,B}\to \mathcal C_{T,B}^{\mathrm{coarse}}$ be the coarse moduli space. The existence of the algebraic space $C_{T,B}^{\mathrm{coarse}}$ was proved by Keel-Mori \cite{KeelMori}. Note that the morphism $\mathcal C_{T,B}\to \mathcal C_{T,B}^{\mathrm{coarse}}$ is quasi-finite and proper  \cite[Thm.~6.12]{Rydh}. 

We now use results of Viehweg \cite[Thm.~1.11]{ViehwegBook} in case $(1)$ and 
Benoist \cite[Cor.~1.2]{BenoistCoarse} in cases $(2)$ and $(3)$ to deduce that $\mathcal C_{T,\CC}^{\mathrm{coarse}}$  
is a quasi-projective scheme over $\CC$ (that the result of \cite{BenoistCoarse} applies in case $(3)$ follows from the classification given in Section \ref{sec:classification}). Since $K\subset \CC$ is flat, the algebraic spaces $(\mathcal C_{T,K}^{\mathrm{coarse}})_{\CC}$ and $\mathcal C_{T,\CC}^{\mathrm{coarse}}$ are naturally isomorphic. Hence by Lemma \ref{lem:quotients}, the algebraic space $\mathcal C_{T,K}^{\mathrm{coarse}}$ is a quasi-projective scheme over $K$. 
Moreover, since the morphism $U_K\to \mathcal C_{T,K}$ is finite and the morphism $\mathcal C_{T,K}\to \mathcal C_{T,K}^{\mathrm{coarse}}$ is quasi-finite and proper, the composed morphism  $U_K\to \mathcal C_{T,K} \to \mathcal C_{T,K}^{\mathrm{coarse}}$ of algebraic spaces is finite. As finite morphisms are quasi-finite separated and $\mathcal C_{T,K}^{\mathrm{coarse}}$ is a scheme, the algebraic space $U_K$ is 
a  scheme by Knutson's criterion \cite[Cor.~II.6.16]{Knutson}. Furthermore, as $U_K\to \mathcal C_{T,K}^{\mathrm{coarse}}$ is a finite morphism of schemes and $\mathcal C_{T,K}^{\mathrm{coarse}}$ is quasi-projective, we can use spreading out \cite[Prop.~4.18]{LMB} to see that $U$ is a smooth finite type  algebraic space over $B$ which is generically quasi-projective. Replacing $B$ by a dense open if necessary, we conclude that $U$ is a  smooth quasi-projective scheme over $B$. 
 
To finish the proof, fix an embedding $K\to \CC$ and consider $U_\CC$. It suffices to show that 
all subvarieties $Z$ of $U_\CC$   are of log-general type, and 
that the complex manifold $U_\CC^{\an}$ is Brody hyperbolic. 
  By Flenner's infinitesimal Torelli theorem \cite[Thm.~3.1]{Fl86} (as used already in Sections \ref{Sec:Torelli} and \ref{sec:automorphisms_faithful}), the period map  of  the polarised variation of Hodge structures on $U_\CC^{\an}$ associated to 
the pull-back $f:X\to U_{\CC}$ to $U_\CC$ of the universal family over $\mathcal C_{T,\CC}$  is an immersion of complex analytic spaces.
 Therefore, the result follows from Lemma \ref{lem:analytic}.
\end{proof}

\begin{remark}
If $T$ is of general type, then to deduce that every subvariety of $U_K$ is of log-general type in the proof of Theorem \ref{thm: finet} one can also appeal to the theorem of  Campana-P{\u{a}}un (\emph{quondam} Viehweg's conjecture) \cite[Cor.~4.6]{CP}.
\end{remark}

\subsection{Lang--Vojta implies Shafarevich}
The aim of this section is to prove Theorem~\ref{thm: lang implies shaf}, via the following
stronger result for arithmetic schemes.

\begin{theorem}\label{thm: lis}
Let $T=(d_1,\ldots,d_c;n)$ be a type and let $B$ be an arithmetic scheme with function field $K$. 
  Assume that one of the following holds.
	\begin{enumerate}
		\item $T$ is of general type.
		\item $3\leq d_1 < d_2 = \cdots =d_c$ and $T \neq (3;1)$.
	\end{enumerate}
Assume Conjecture \ref{lang conj}. Then the set of $B$-linear isomorphism classes of smooth complete intersections of type $T$ over $B$ is finite. 
\end{theorem}
\begin{proof} 
By the uniqueness of good models (Lemma~\ref{lem: un}), to prove the theorem we are free to replace $B$ 
by any non-empty open subset. Moreover, given any finite map of arithmetic schemes $B' \to B$ and a smooth
complete intersection $X$ of type $T$ over $B'$,
Lemma~\ref{lem: un} and Theorem \ref{thm: twists}  imply that there are only finitely many $B$-linear isomorphism classes of
smooth complete intersections of type $T$ over $B$ which become $B'$-linearly isomorphic
to $X$ over $B'$. Hence, on replacing $B$ by a finite cover if necessary,
by Theorem \ref{thm: finet} we may assume that there exists a finite \'etale atlas $U\to \mathcal C_{T,B}$, where $U \to B$ is a smooth quasi-projective $B$-scheme whose generic fibre has the property that all its subvarieties are of log-general type.

We now apply a descent argument in the spirit of the proof of the theorem of Chevalley-Weil \cite[\S4.2]{Ser97}. By pull-back, any $B$-point of $\mathcal C_{T,B}$ induces a $\widetilde B$-point of $U$, where $\widetilde B\to B$ is some finite \'etale morphism whose degree is at most the  degree of $U\to \mathcal C_{T,B}$.
As $\widetilde B\to B$ is finite \'etale, the connected components of the scheme $\widetilde B$ are arithmetic schemes. Therefore, the set $U(\widetilde B)$ is finite  by Conjecture \ref{lang conj}. Moreover, the set of $B$-isomorphism classes of finite \'etale morphisms $\widetilde B\to B$ of bounded degree is finite by Hermite-Minkowski (Theorem \ref{thm:HM}). Hence the set $[\mathcal C_T(B)]$ is finite. In particular, there are only finitely many $B$-linear isomorphism classes of smooth complete intersections of type $T$ over $B$, as required.
\end{proof}

\begin{proof}[Proof of Theorem \ref{thm: lang implies shaf}]
Let $T$ be a type which is either of general type or is a hypersurface. 
Let $K$ be a number field and let $B \subset \Spec \OO_K$ be a dense open subset.
By Theorem \ref{theorem: main theorem},
we may assume that $T$ has Hodge level at least $2$. Moreover, replacing $B$ by a dense open if necessary,
we may  assume that $\Pic(B) =0$. In particular, by Lemma \ref{lem:glueing}, any smooth complete intersection of type $T$ over $K$ with good reduction over $B$ has a good model over $B$.  However, assuming the Lang--Vojta conjecture, Theorem \ref{thm: lis} implies that the set of $B$-linear isomorphism classes of such models is finite, which proves the result.
\end{proof}
%\bibliography{refsci}{}
%\bibliographystyle{plain}

\def\cprime{$'$}

\end{document}